\documentclass[12pt,english]{amsart}

\usepackage[inner=1in,outer=1in,top=1in,bottom=1in]{geometry}

\usepackage[T1]{fontenc}
\usepackage{amssymb}

\usepackage{paralist}
\usepackage{babel}
\usepackage[pdftex]{graphicx}    
\usepackage[leqno]{amsmath}
\usepackage{amsthm}

\newcommand\mbb{\mathbb}

\newcommand\mcal{\mathcal}

\newcommand\ol{\overline}

\newcommand\sV{\mcal{V}}

\newcommand\C{\mbb{C}}
\newcommand\N{\mbb{N}}
\renewcommand\P{\mbb{P}}
\newcommand\Q{\mbb{Q}}
\newcommand\R{\mbb{R}}

\DeclareMathOperator*\rank{rank}
\DeclareMathOperator*\trace{tr}

\newcommand\isom{\cong}
\renewcommand\epsilon{\varepsilon}
\renewcommand\ge{\geqslant}
\renewcommand\le{\leqslant}
\renewcommand\phi{\varphi}
\renewcommand\theta{\vartheta}

\theoremstyle{plain}
\newtheorem{Thm}{Theorem}
\newtheorem{Prop}[Thm]{Proposition}
\newtheorem{Cor}[Thm]{Corollary}
\newtheorem{Lemma}[Thm]{Lemma}

\newtheorem*{Thm*}{Theorem}
\newtheorem*{Prop*}{Proposition}
\newtheorem*{Cor*}{Corollary}
\newtheorem*{Lemma*}{Lemma}
\newtheorem*{Conjecture*}{Conjecture}

\theoremstyle{definition}

\newtheorem{Def}[Thm]{Definition}

\newtheorem{Example}[Thm]{Example}

\newtheorem{Remark}[Thm]{Remark}

\newtheorem*{Constr*}{Construction}
\newtheorem*{Def*}{Definition}
\newtheorem*{Defs*}{Definitions}
\newtheorem*{Example*}{Example}
\newtheorem*{Examples*}{Examples}
\newtheorem*{Exercise*}{Exercise}
\newtheorem*{LemmaDef*}{Lemma and Definition}
\newtheorem*{Notation*}{Notation}
\newtheorem*{Problem*}{Problem}
\newtheorem*{Question*}{Question}
\newtheorem*{Remark*}{Remark}
\newtheorem*{Remarks*}{Remarks}
\newtheorem*{Warning*}{Warning}

\numberwithin{equation}{section}
\numberwithin{Thm}{section}

\newcommand\varx{x}

\begin{document}
\title[Hyperbolic polynomials, interlacers, and sums of squares]{Hyperbolic polynomials, interlacers, and sums of squares}


 \author{Mario Kummer}
\address{Universit\"at Konstanz, Germany} 
\email{Mario.Kummer@uni-konstanz.de}

\author{Daniel Plaumann}
\address{Universit\"at Konstanz, Germany} 
\email{Daniel.Plaumann@uni-konstanz.de}

 \author{Cynthia Vinzant}
\address{University of Michigan, Ann Arbor, MI, USA}
\email{vinzant@umich.edu}

\maketitle

\begin{abstract}
  Hyperbolic polynomials are real polynomials whose real hypersurfaces
  are maximally nested ovaloids, the innermost of which is convex. These
  polynomials appear in many areas of mathematics, including
  optimization, combinatorics and differential equations.  Here we
  investigate the special connection between a hyperbolic polynomial
  and the set of polynomials that interlace it.  This set of
  interlacers is a convex cone, which we write as a linear slice of the
  cone of nonnegative polynomials. In particular, this allows us to realize
  any hyperbolicity cone as a slice of the cone of nonnegative
  polynomials. Using a sums of squares relaxation, we then approximate
  a hyperbolicity cone by the projection of a spectrahedron.  A
  multiaffine example coming from the V\'amos matroid shows that this
  relaxation is not always exact. 
  Using this theory, we characterize the 
  real stable multiaffine polynomials that have a definite determinantal representation
  and construct one when it exists. 
\end{abstract}

\section{Introduction}

  A homogeneous polynomial $f\in \R[\varx]$ of degree $d$ in variables
  $x=(x_1,\dots,x_n)$ is called \textbf{hyperbolic} with respect to a point
  $e\in \R^{n}$ if $f(e)\neq 0$ and for every $a\in \R^{n}$, all
  roots of the univariate polynomial $f(te+a)\in \R[t]$ are real.  Its
  \textbf{hyperbolicity cone}, denoted $C(f,e)$ is the connected
  component of $e$ in $\R^{n} \backslash \mathcal{V}_{\R}(f)$ and
  can also be defined as
\[ C(f,e) = \{ a\in \R^{n} \;\colon\; f(te-a) \neq 0 \;\text{ when }\; t\leq 0\}.\]

As shown in G\r{a}rding \cite{Gar}, $C(f,e)$ is an open convex cone
and $f$ is hyperbolic with respect to any point contained in it.
Hyperbolicity is reflected in the topology of the real projective variety
$\sV_\R(f)$ in $\P^{n-1}(\R)$.  If $\sV_{\R}(f)$ is smooth, then $f$ is hyperbolic if and only if
$\sV_\R(f)$ consists of $\lfloor\frac{d}{2}\rfloor$ nested ovaloids, and
a pseudo-hyperplane if $d$ is odd (see \cite[Thm.~5.2]{HV}).

\begin{figure}[h]
 \includegraphics[width=3cm]{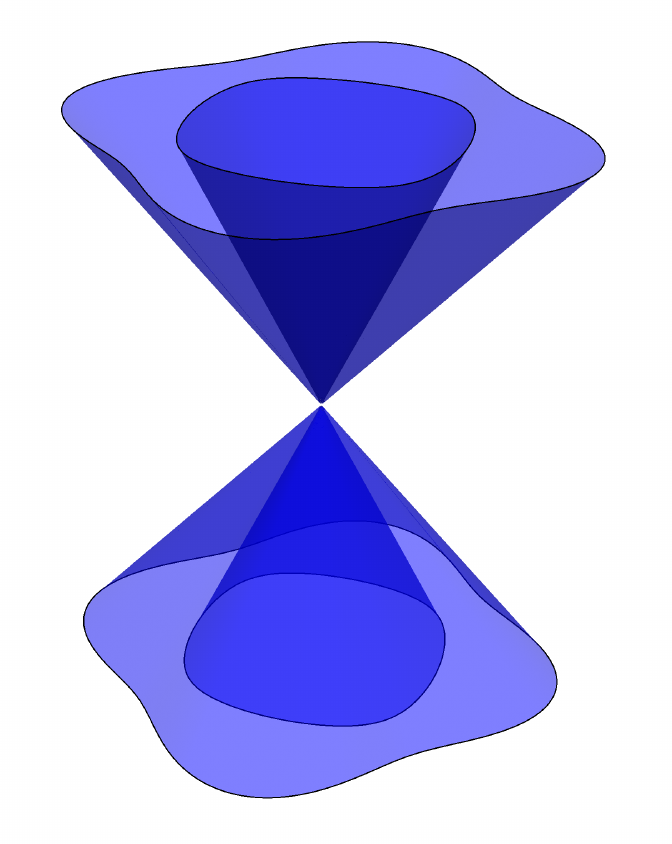} \quad
 \includegraphics[width=3.5cm]{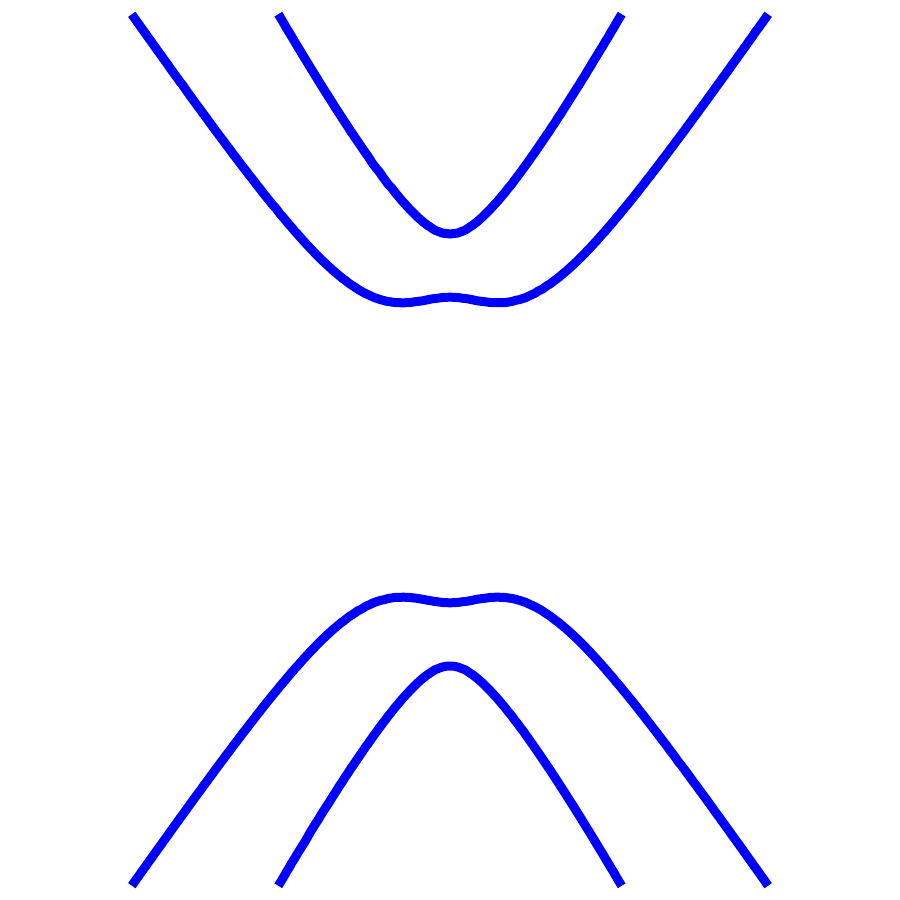}\quad
  \includegraphics[width=3.5cm]{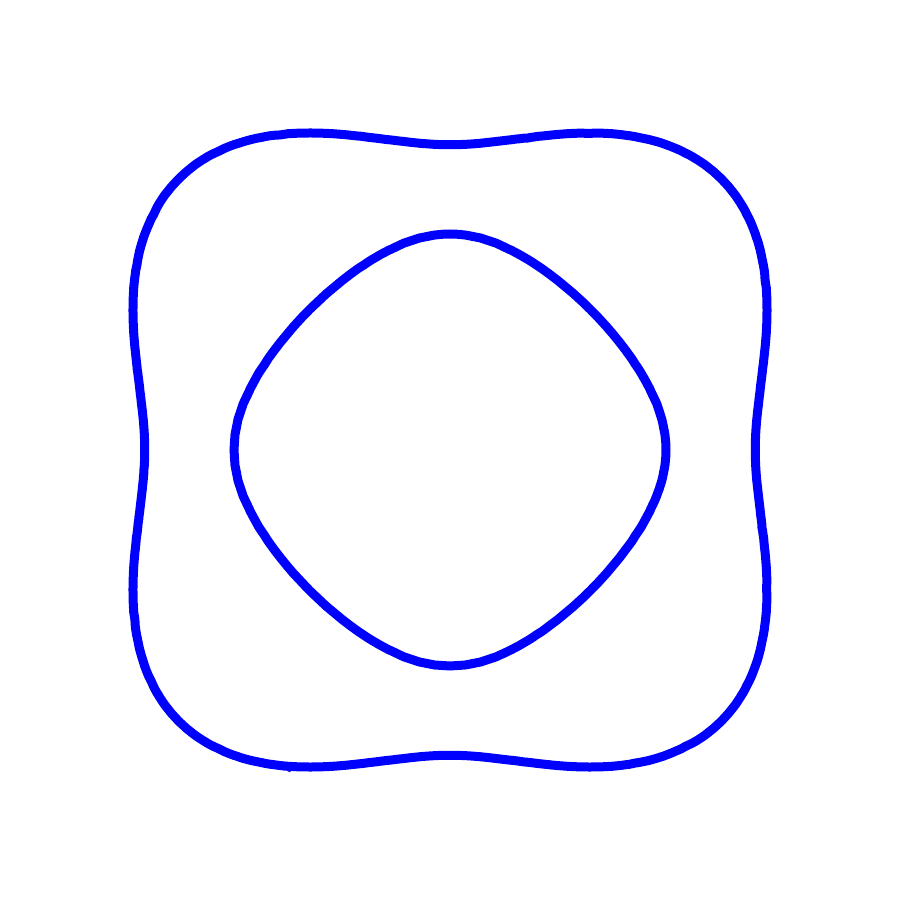}
\caption{A quartic hyperbolic hypersurface and two of its affine slices.  }
\label{fig:cubic}
\end{figure}
 
A hyperbolic program, introduced and developed by G\"uler
\cite{Gue}, Renegar \cite{Ren} and others, is the problem of maximizing a linear function over an 
affine section of the convex cone $C(f,e)$.  This provides a very general 
context in which interior point methods are effective.  For example, taking
$f = \prod_i x_i$ and $e=(1,\hdots, 1)$, we see that $C(f,e)$ is the positive
orthant $(\R_+\!)^{n}$ and the corresponding hyperbolic program 
is a linear program.  If instead we take $f$ as the determinant 
of a symmetric matrix of variables $X=(x_{ij})$ and $e$ is the identity 
matrix, then $C(f,e)$ is the cone of positive definite matrices.
\begin{center}
\begin{tabular}{c|c|c|c}
$f$ & $e$ & $C(f,e)$ & hyperbolic program\raisebox{0pt}[15pt][10pt]{} \\
\hline
$\prod_i x_i$ & $\;\;(1,\hdots, 1)\;\;$  & $(\R_+)^{n}$ & linear program \raisebox{0pt}[15pt][10pt]{}\\

$\;\det(X)\;$& $I$ & $\;$positive definite matrices$\;$ & $\;$semidefinite program$\;$\raisebox{0pt}[15pt][10pt]{}
\end{tabular}
\end{center}\smallskip

It is a fundamental open question whether or not every hyperbolic
program can be rewritten as a semidefinite program.  Helton and
Vinnikov \cite{HV} showed that if $f\in \R[x_1,x_2,x_3]$ is hyperbolic
with respect to a point $e$, then $f$ has a \textit{definite
  determinantal representation} $f = \det(\sum_ix_iM_i)$ where $M_1,
M_2, M_3$ are real symmetric matrices and the matrix $\sum_i e_i M_i$ 
is positive definite. Thus every three dimensional hyperbolicity cone is a slice of the
cone of positive semidefinite matrices. For a survey of these results
and future perspectives, see also \cite{Vin}. On the other hand,
Br\"and\'en \cite{Bra11} has given an example of a hyperbolic
polynomial $f$ (see Example~\ref{ex:Vamos}) such that no power of $f$
has a definite determinantal representation. There is a close
connection between definite determinantal representations of a
hyperbolic polynomial $f$ and polynomials of degree one-less that
\textit{interlace} it, which has also been used in \cite{PV} to study
Hermitian determinantal representations of hyperbolic curves.
 
\begin{Def}
  Let $f,g\in\R[t]$ be univariate polynomials with only real zeros and with $\deg(g)=\deg(f)-1$. Let
  $\alpha_1\le\cdots\le\alpha_d$ be the roots of $f$, and let
  $\beta_1\le\cdots\le\beta_{d-1}$ be the roots of $g$. We say that
  \textbf{$g$ interlaces $f$} if $\alpha_i\le\beta_i\le\alpha_{i+1}$ for all
  $i=1,\dots,d-1$. If all these inequalities are strict, we say that
  \textbf{$g$ strictly interlaces $f$}.

  If $f\in \R[\varx]$ is hyperbolic with respect to $e$ and $g$ is homogeneous of
  degree $\deg(f)-1$, we say that \textbf{$g$ interlaces $f$ with
    respect to $e$} if $g(te+a)$ interlaces $f(te+a)$ for every
  $a\in\R^n$. This implies that $g$ is also hyperbolic with
  respect to $e$. We say that $g$ \textbf{strictly interlaces $f$} if
  $g(te+a)$ strictly interlaces $f(te+a)$ for $a$ in a nonempty Zariski-open
  subset of $\R^n$.
\end{Def}

The most natural example of an interlacing polynomial is the derivative. If $f(t)$ is a real polynomial 
with only real roots, then its derivative $f'(t)$ has only real roots, which interlace the roots of $f$. 
Extending this to multivariate polynomials, we see that the roots of $\frac{\partial}{\partial t}f(te+a)$ interlace
those of $f(te+a)$ for all $a\in \R^{n}$. Thus 
\[ D_ef \;\; = \;\; \sum_{i=1}^n e_i \frac{\partial f}{\partial x_i}\]
interlaces $f$ with respect to $e$. If $f$ is square-free, then $D_ef$
strictly interlaces $f$. This was already noted by G\r{a}rding
\cite{Gar} and has been used extensively, for example in \cite{Bra07}
and \cite{Ren}; for general information on interlacing polynomials,
see also \cite{Fis} and \cite[Ch. 6]{MR1954841}.

 \begin{figure}[h]
 \includegraphics[width=4cm]{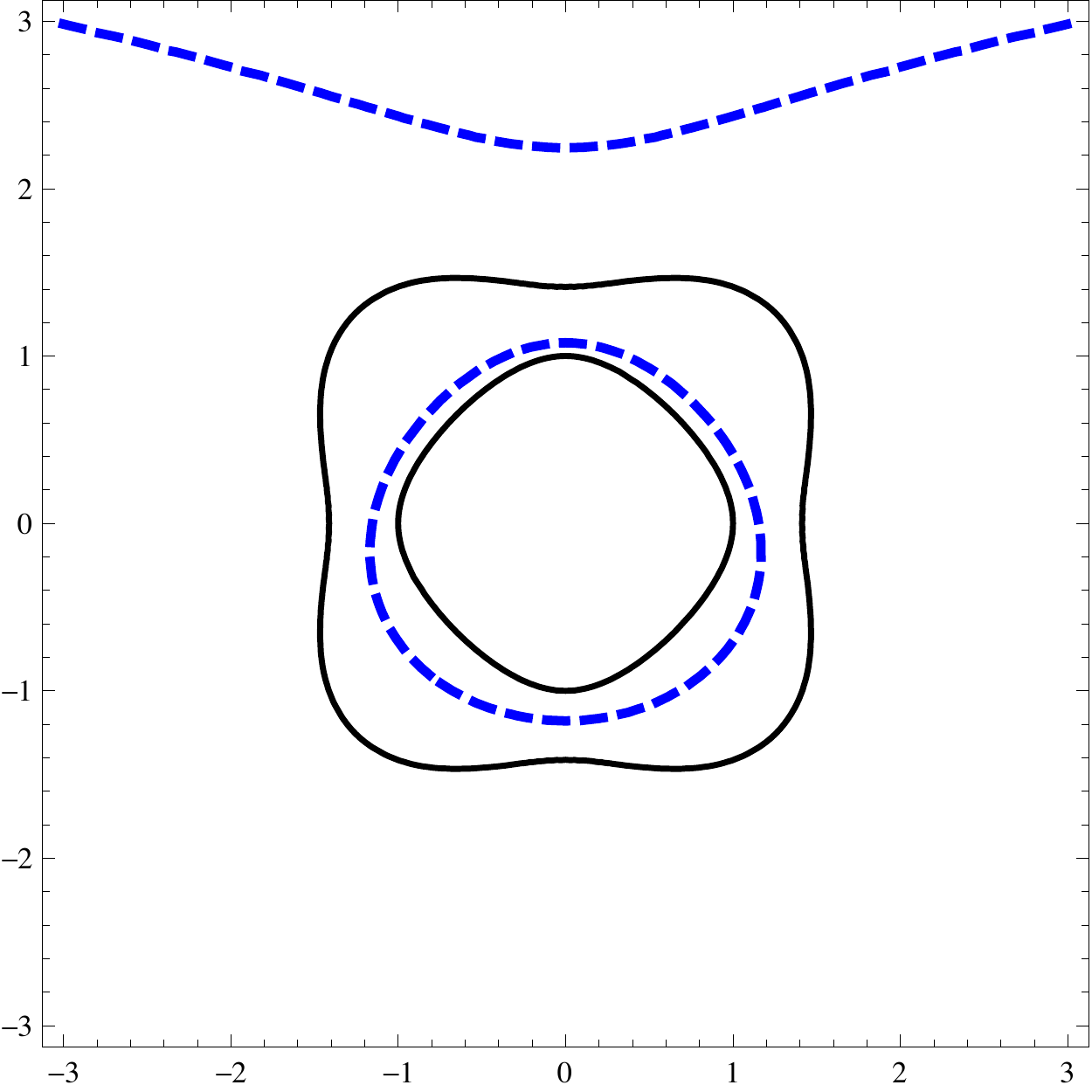} \quad
 \includegraphics[width=4cm]{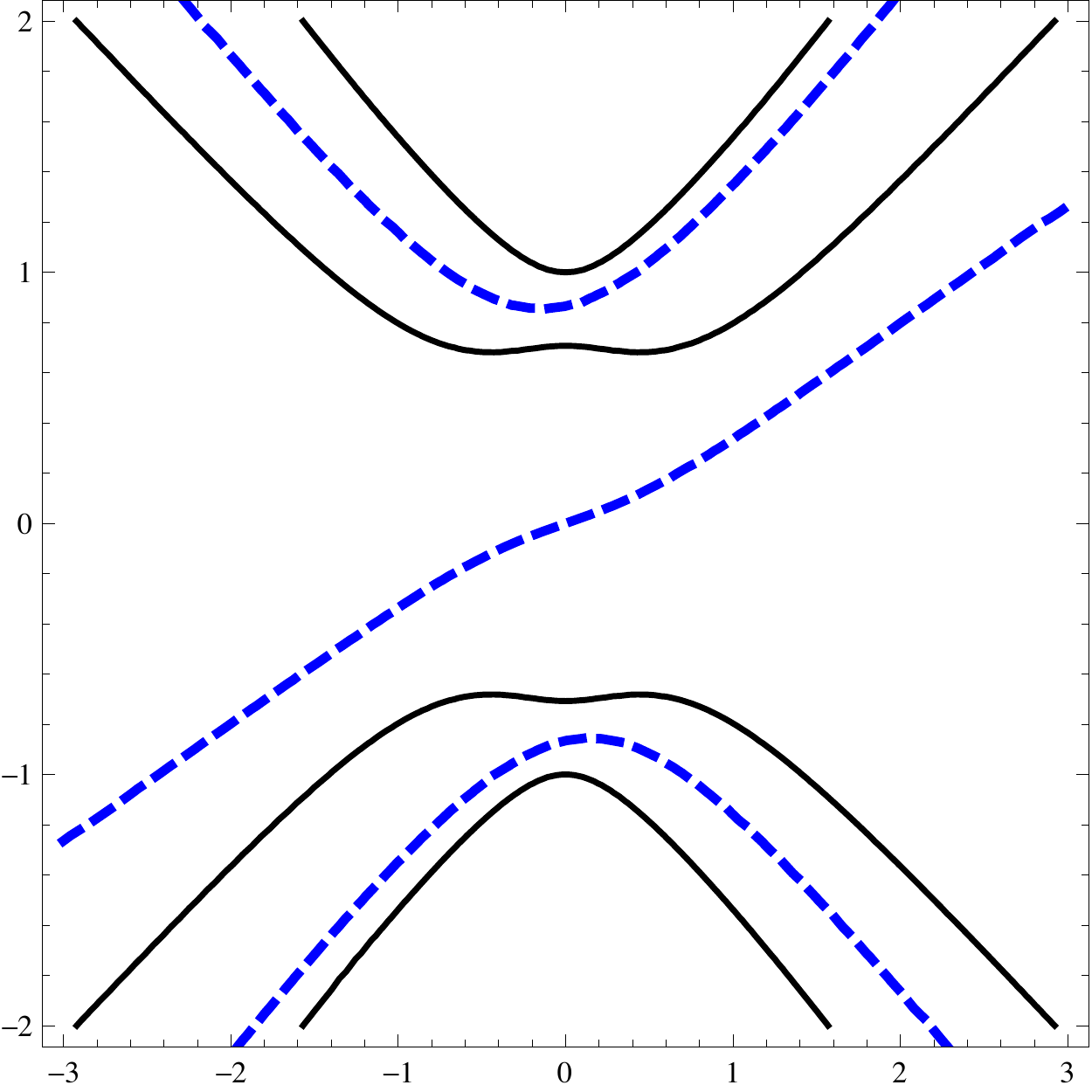} 
 \vspace{-.2cm}
\caption{Two affine slices of a cubic interlacing a quartic.  }
\label{fig:cubic}
\end{figure}

\begin{Remark}
  If $f$ is square-free and $d=\deg(f)$, then $f$ is hyperbolic with respect to $e$ if and
  only if $f(te+a)$ has $d$ distinct real roots for $a$ in a Zariski-open
  subset of $\R^{n}$.  In this case, if $g$ interlaces $f$ and has no common factors with $f$, 
  then $g$ strictly interlaces $f$. 
\end{Remark}

In this paper, we examine the set of polynomials in $\R[x]_{d-1}$
interlacing a fixed hyperbolic polynomial. The main result is a
description of a hyperbolicity cone $C(f,e)$ as a linear slice of the
cone of nonnegative polynomials. Using the cone of sums of
squares instead gives an inner approximation of $C(f,e)$ by a projection of a
spectrahedron.  This is closely related to recent results due to
Netzer and Sanyal \cite{NS} and Parrilo and Saunderson \cite{PS}. We
discuss both this theorem and the resulting approximation in
Section~\ref{sec:hypCone}. In Section~\ref{sec:HermDets} we see that
the relaxation we obtain is exact if some power of $f$ has a definite
determinantal representation.  A multiaffine example for which our
relaxation is not exact is discussed in Section~\ref{sec:multiaff}.
Here we also provide a criterion to test whether or not a hyperbolic multiaffine 
polynomial has a definite determinantal representation.
The full cone of interlacers has a nice structure, which we discuss in
Section~\ref{sec:boundary}. First we need to build up some basic facts
about interlacing polynomials.\\

\textbf{Acknowledgements.} We would like to thank Alexander Barvinok,
Petter Br\"and\'en, Tim Netzer, Rainer Sinn, and Victor Vinnikov for
helpful discussions on the subject of this paper.  Daniel Plaumann was
partially supported by a Feodor Lynen return fellowship of the
Alexander von Humboldt-Foundation. Cynthia Vinzant was partially
supported by the National Science Foundation RTG grant DMS-0943832 
and award DMS-1204447.

\section{Interlacers}\label{sec:Interlacers}

Let $f$ be a homogeneous polynomial of degree $d$ that is hyperbolic
with respect to the point $e\in \R^{n}$.  We will always assume that
$f(e)>0$.  Define ${\rm Int}(f,e)$ to be the set of real polynomials of degree $d-1$ that
interlace $f$ with respect to $e$ and are positive at $e$:
\[
 {\rm Int}(f,e) \;=\;  \bigl\{ g\in \R[\varx]_{d-1}
\;\colon\;\text{$g$ interlaces $f$ with respect to }e\text{ and
}g(e)>0\bigr\}.
\]
As noted above, the hyperbolicity cone $C(f,e)$ depends only on $f$ and
the connected component of $\R^n\setminus\sV_\R(f)$ containing $e$. In
other words, we have $C(f,e)=C(f,a)$ for all $a\in C(f,e)$. We will see shortly that ${\rm Int}(f,e)$ does not
depend on $e$ either, but only on $C(f,e)$.

\begin{Thm} \label{thm:Interlacers} Let $f \in \R[\varx]_d$ be square-free and hyperbolic with respect to $e\in \R^{n}$, 
where $f(e)>0$. 
For $h \in \R[\varx]_{d-1}$, the following are equivalent:
\begin{enumerate}
\item $h \in  {\rm Int}(f,e)$;
\item $h \in  {\rm Int}(f,a)$ for all $a \in C(f,e)$;
\item $ D_ef\cdot h$ is nonnegative on $\mathcal{V}_{\R}(f)$;
\item $D_ef \cdot h- f \cdot D_eh $ is nonnegative on $\R^{n}$.
\end{enumerate}
\end{Thm}

The proof of this theorem and an important corollary are at the end of this section. First, we need 
to build up some theory about the forms in ${\rm Int}(f,e)$. 

\begin{Lemma} \label{lem:commonFactors}
Suppose $f_1$, $f_2$, and $h$ are real homogeneous polynomials.  
\begin{enumerate}[a{\rm )}]
\item The product $f_1\cdot f_2$ is hyperbolic with respect to $e$ if and only if both $f_1$ and $f_2$ are hyperbolic with respect to $e$. In this case, 
$C(f_1\cdot f_2 , e) = C(f_1,e)\cap C(f_2,e)$.  
\item If $f_1$ and $f_2$ are hyperbolic with respect to $e$, then $f_1\cdot h$ interlaces $f_1\cdot f_2$ if and only if
$h$ interlaces $f_2$. 
\item If $h$ interlaces $(f_1)^k f_2$ for $k\in \mathbb{N}$, then $(f_1)^{k-1}$ divides $h$. 
\end{enumerate}
\end{Lemma}

\begin{proof} 
These statements are checked directly after reducing to the one-dimensional
case.
\end{proof}

\begin{Lemma}\label{lem:samesign}
For any $g$ and $h$ in ${\rm Int}(f,e)$, the product $g\cdot h$ is nonnegative on $\mathcal{V}_{\R}(f)$.  
\end{Lemma}
\begin{proof}
To prove this statement, it suffices to restrict to any line $x = te+a$ where $a\in \R^{n}$. Suppose that
$f(te+a) \in \R[t]$ has roots $\alpha_1\leq \cdots \leq  \alpha_d$ and $g(te+a)$ and $h(te+a)$
have roots $\beta_1\leq \cdots \leq \beta_{d-1}$ and  $\gamma_1\leq \cdots \leq  \gamma_{d-1}$, respectively. 
By the assumption that both $g$ and $h$ interlace $f$, we know that 
$\beta_i, \gamma_i\in [\alpha_i,\alpha_{i+1}]$ for all $1\leq i\leq d-1$. 
Thus, if $\alpha_i$ and $\alpha_j$ are not also roots of $g(te+a)$ or $h(te+a)$, the polynomial
$g(te+a)h(te+a)$ has an even number of roots in the interval $[\alpha_i, \alpha_j]$. 
Then the sign of $g(\alpha_ie+a)h(\alpha_ie+a)$ is the same for all $i$ for which it is not zero. 
Because $g(e)h(e)>0$, we see that that sign must be nonnegative. 
\end{proof}

 \begin{figure}[h]
   \includegraphics[width=5cm]{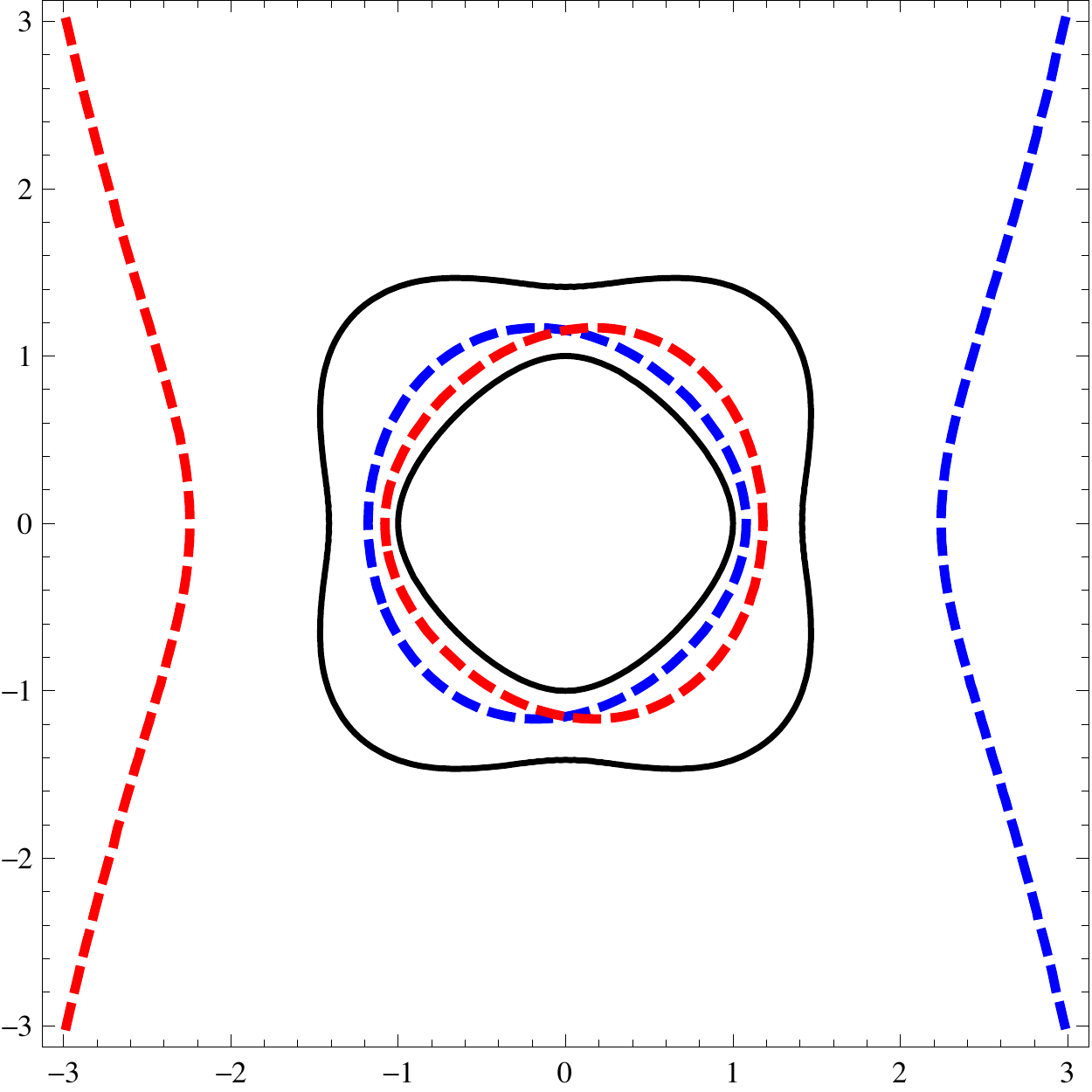} \hspace{5em}
 \includegraphics[width=5cm]{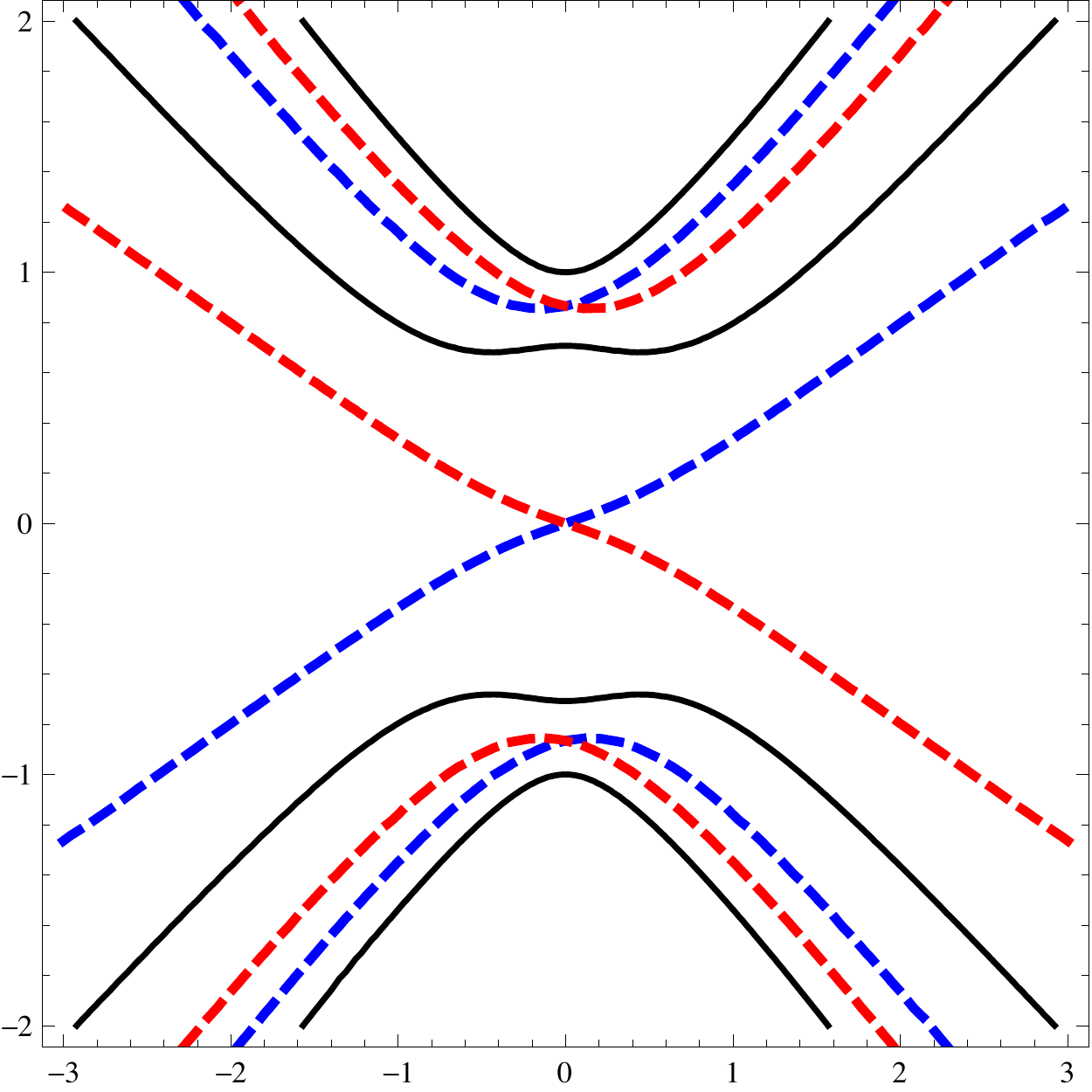}
\caption{Affine slices of two cubics interlacing a quartic. }
\label{fig:cubic}
\end{figure}

\begin{Lemma} \label{lem:nonnegInterlacers}
Suppose that $f$ is square-free and that $g \in {\rm Int}(f,e)$ \textit{strictly} interlaces $f$. 
Then a polynomial $h\in \R[\varx]_{d-1}$ belongs to ${\rm Int}(f,e)$ if and only if $g\cdot h$ is nonnegative on $\mathcal{V}_{\R}(f)$. 
\end{Lemma}

\begin{proof} One direction follows from Lemma~\ref{lem:samesign}.
  For the other, let $h\in \R[\varx]_{d-1}$ for which $g\cdot h$ is
  nonnegative on $\mathcal{V}_\R(f)$. First, let us consider the case
  where $f$ and $h$ have no common factor.  Then, for generic $a \in
  \R^n$, the roots of $f(te+a)$ are distinct from each other and from
  the roots of $g(te+a)$ and $h(te+a)$.  The product $g(te+a)h(te+a)$
  is then positive on all of the roots of $f(te+a)$. Since $g(te+a)$
  changes sign on consecutive roots of $f(te+a)$, we see that
  $h(te+a)$ must have a root between each pair of consecutive roots of
  $f(te+a)$, and thus $h$ interlaces $f$ with respect to $e$.

  Now suppose $f=f_1\cdot f_2$ and $h=f_1\cdot h_1$. We will show that
  $h_1$ interlaces $f_2$, and thus $h$ interlaces $f$.  Again, we can
  choose generic $a$ for which the roots of $f(te+a)$ and $g(te+a)$
  are all distinct.  Consider two consecutive roots
  $\alpha<\beta$ of the polynomial $f_2(te+a)$. Let $k$ be the
  number of roots of $f_1(te+a)$ in the interval $(\alpha,
  \beta)$.  Because $g$ strictly interlaces $f=f_1\cdot f_2$, its
  restriction $g(te+a)$ must have $k+1$ roots in the interval
  $(\alpha, \beta)$.  Thus the polynomial $g(te+a)f_1(te+a)$ has
  an odd number of roots in this interval and must therefore have
  different signs in $\alpha$ and $\beta$. Since $g\cdot f_1
  \cdot h_1 \geq 0$ on $V(f)$, the polynomial $h_1(te+a)$ must have a
  root in this interval. Thus $h_1$ interlaces $f_2$ and $h$
  interlaces $f$.
\end{proof}

\begin{Example} In the above lemma, it is indeed necessary that $f$
  and $g$ be without common factors. For example, consider
  $f=(x^2+y^2-z^2)(x-2z)$ and $g=(x^2+y^2-z^2)$.  Both $f$ and $g$ are
  hyperbolic with respect to $[0:0:1]$ and $g$ interlaces $f$ with
  respect to this point. However if $h=y(x-2z)$, then $g\cdot h$
  vanishes identically on $\mathcal{V}_{\R}(f)$ but $h$ does not
  interlace $f$.
\end{Example}

  For $a\in C(f,e)$, the derivative $D_af$ obviously
  interlaces $f$ with respect to $a$, since $f$ is hyperbolic with respect
  to $a$. We need to show that $D_af$ also interlaces $f$ with
  respect to $e$. 

\begin{Lemma}\label{lem:main}
For $a\in C(f,e)$, the polynomial $D_e f\cdot D_a f$ is nonnegative on $\mathcal{V}_{\R}(f)$. 
\end{Lemma}
\begin{proof} 
 For any $b\in C(f,e)$ and $x\in \mathcal{V}_{\R}(f)$, let
  $\alpha_1(b,x)\leq \cdots \leq \alpha_d(b,x)$ denote the roots of
  $f(tb +x)$.  Because $C(f,e)$ is convex, the line segment joining
  $e$ and $a$ belongs to this cone.  As we vary $b$ from $e$ to $a$
  along this line segment, the roots $\{\alpha_i(b,x)b +x\}_{i\in
    [d]}$, form $d$ \textit{non-crossing} arcs in the plane $x+{\rm
    span}\{e,a\}$, as shown in Figure~\ref{fig:arcs}.  Since $f(x)=0$,
  one of these arcs is just the point $x$. That is, there is some $k$
  for which $\alpha_k(b,x)=0$ for all $b$ in the convex hull of $e$
  and $a$.

Now $f(e)>0$ implies $f(b)>0$ for all $b\in C(f,e)$. Thus $\frac{\partial}{\partial t}f(tb+x)$ is positive 
for $t>\alpha_d(b,x)$. Furthermore, the sign of this derivative on the $i$th root, $\alpha_i(b,x)$ depends only on $i$. Specifically, 
for all $i=1,\hdots, d$, 
\[ (-1)^{d-i}\cdot D_b f(\alpha_{i}(b,x)b+x) \; \geq \;0. \]
In particular, the sign of $D_bf$ on the $k$th root, $\alpha_k(b,x)b+x = x$, is constant:
\[(-1)^{d-k}D_bf(x) \geq 0.\]
Then, regardless of $k$, the product $D_ef(x)D_af(x)$ is non-negative. 
\end{proof}

 \begin{figure}
 \includegraphics[width=6cm]{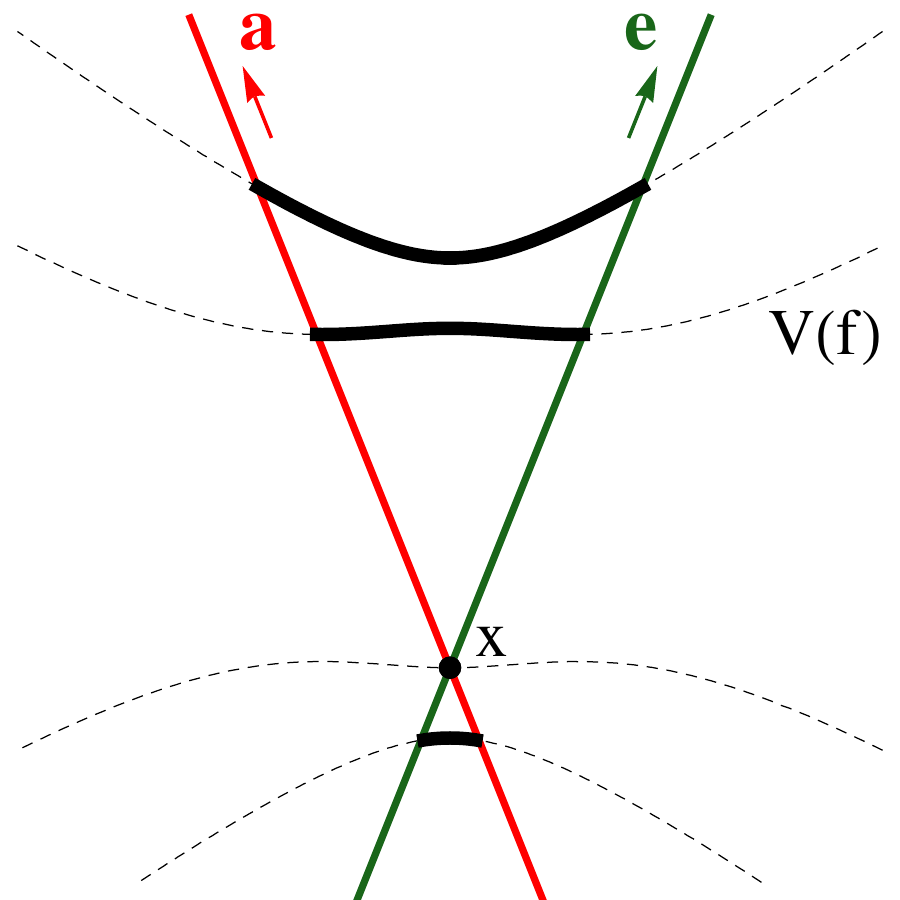} 
\caption{Non-crossing arcs of Lemma~\ref{lem:main}. }
\label{fig:arcs}
\end{figure}

Now we are ready to prove Theorem~\ref{thm:Interlacers}.

\begin{proof}[Proof of Theorem~\ref{thm:Interlacers}] $\;$ \smallskip

 ($4 \Rightarrow 3$) Clear. \medskip

($1 \Leftrightarrow 3$) If $f$ is square free, then $D_ef$ strictly interlaces $f$. 
This equivalence then follows from Lemma~\ref{lem:nonnegInterlacers}.\medskip

($1,3 \Rightarrow 4$) Here we need a useful fact about Wronskians. 
The Wronskian of two univariate polynomials $p(t),q(t)$ is the polynomial
\[ W(p,q) \; = \; p\cdot q' - p' \cdot q \; = \; q^2 \cdot \left(
  \frac{p}{q}\right)'.\] It is a classical fact that if the roots of
$p$ and $q$ are all distinct and interlace, then $W(p,q)$ is a
nonnegative or nonpositive polynomial \cite[\S2.3]{Wag}.  Thus
if $h\in {\rm Int}(f,e)$ is coprime to $f$, then for generic $x$, the roots of
$f(te+x)$ and $h(te+x)$ interlace and are all distinct.  Thus their
Wronskian $ h(te+x) f'(te+x)- h'(te+x) f(te+x)$ is either nonnegative
or nonpositive for all $t$.  By $(3)$, the product $h(te+x)f'(te+x)$
is nonnegative on the zeroes of $f$, so we see that the Wronskian is
nonnegative.  Setting $t=0$ gives us that $h\cdot D_ef - D_eh\cdot f $
is nonnegative for all $x\in \R^n$, as desired. If $f$ and $h$
share a factor, say $f=f_1\cdot f_2$, $h=f_1\cdot h_1$, we can use the identity
$W(f_1\cdot f_2,f_1\cdot h_1)=f_1^2W(f_2,h_1)$ to reduce to the coprime case. \medskip

($2 \Leftrightarrow 1$)
Because $f$ is square free, both $D_e f$ and $D_af$ share no factors with $f$. Thus $D_ef$ strictly interlaces
$f$ with respect to $e$ and $D_af$ strictly interlaces $f$ with respect to $a$.  

Suppose $h$ interlaces $f$ with respect
  to $a$ and $h(a)>0$. 
  By Lemma~\ref{lem:nonnegInterlacers}, $h\cdot D_af$ is nonnegative
  on $\mathcal{V}_{\R}(f)$. Using Lemma~\ref{lem:main}, we see that $D_ef
  \cdot D_af$ is also nonnegative on $\mathcal{V}_{\R}(f)$.  Taking
  the product, it follows that $(D_af)^2\cdot D_ef \cdot h $ is
  nonnegative on $\mathcal{V}_{\R}(f)$.  Because 
  $D_af $ and $f$ have no common factors, we can conclude that
  $D_ef \cdot h$ is nonnegative on $\mathcal{V}_{\R}(f)$. Using
  Lemma~\ref{lem:nonnegInterlacers} again, we have $h\in {\rm Int}(f,e)$.
  Switching the roles of $a$ and $e$ in this argument gives the reverse implication. 
\end{proof}

\begin{Cor}\label{cor:InterlacersNon} The set ${\rm Int}(f,e)$ is a closed
  convex cone.  If $f$ is square-free, this cone is linearly
  isomorphic to a section of the cone of nonnegative polynomials of
  degree $2\deg(f)-2$:
\begin{equation}\label{eq:IntNonneg2}
 {\rm Int}(f,e) \;=\; \{ h\in \R[\varx]_{\deg(f)-1} \;:\; D_ef \cdot h - f \cdot D_eh \;\geq 0 \;\text{ on } \R^{n} \}.
 \end{equation}
 If $f=f_1\cdot f_2$ where $\mathcal{V}(f)=\mathcal{V}(f_2)$ and $f_2$ is square-free, 
 then 
\[{\rm Int}(f,e) \;=\; f_1 \cdot {\rm Int}(f_2, e)\]
and is isomorphic to a section of the cone of nonnegative polynomials 
of degree $2\deg(f_2)-2$.
\end{Cor}

\begin{proof}
For square-free $f$, the description \eqref{eq:IntNonneg2} follows directly from Theorem~\ref{thm:Interlacers}.  
The map \[h \;\; \mapsto\;\;  D_ef \cdot h - f \cdot D_eh\]  
is a linear map from $\R[\varx]_{\deg(f)-1}$ to $\R[\varx]_{(2\deg(f)-2)}$. We see that  ${\rm Int}(f,e)$ is the 
preimage of the cone of nonnegative polynomials in $\R[\varx]_{(2\deg(f)-2)}$ under this map. 
We can also check that this map is injective. Because $f$ is square free, $D_ef$ and $f$ are coprime. Hence 
if $f$ were to divide $D_ef \cdot h$, then $f$ would have to divide $h$, which it cannot. Thus $D_ef \cdot h - f \cdot D_eh$
cannot be identically zero. 

If $f$ is not square-free, then $f$ factors as $f_1\cdot f_2$ as above. By Lemma~\ref{lem:commonFactors}(c),
any polynomial that interlaces $f$ must be divisible by $f_1$. By part (b), the remainder must interlace $f_2$. Thus
${\rm Int}(f,e) \subseteq f_1 \cdot {\rm Int}(f_2, e)$.  Similarly, if $h$ interlaces $f_2$, then $f_1\cdot h$ interlaces $f=f_1\cdot f_2$. 
Thus ${\rm Int}(f,e)$ is the image of the convex cone ${\rm Int}(f_2,e)$ under a linear map, namely multiplication by $f_1$. This 
shows that it is linearly isomorphic to a section of the cone of nonnegative polynomials of degree $2\deg(f_2)-2$.
\end{proof}


\section{Hyperbolicity cones and Nonnegative polynomials} \label{sec:hypCone}

An interesting consequence of the results in the preceding section is 
that we can recover the hyperbolicity cone $C(f,e)$ as a linear section of ${\rm Int}(f,e)$, and thus as a 
linear section of the cone of nonnegative polynomials. 
We show this by considering which partial derivatives $D_a(f)$ interlace $f$. Using
 Theorem~\ref{thm:Interlacers}, we often have to deal with the polynomials 
\[\Delta_{e,a}f \;\;=\;\;  D_e f \cdot D_a f - f\cdot D_eD_a f.
\]
Notice that $\Delta_{e,a}f$ is homogeneous of 
degree $2d-2$, symmetric in $e$ and $a$, and linear in each.

\begin{Thm} \label{thm:niceCones}
Let $f\in \R[\varx]_d$ be square-free and hyperbolic with respect to the point 
$e\in \R^{n}$. The intersection of ${\rm Int}(f,e)$ with the plane spanned by the partial derivatives of $f$
is the image of $\overline{C(f,e)}$ under the linear map $a \mapsto D_a f$. That is, 
\begin{equation}\label{eq:slicedInt}
\overline{C(f,e)} \;\; = \;\; \{ a\in \R^{n} \;:\; D_a f \;\in\; {\rm Int}(f,e) \}.
\end{equation}
Furthermore, $\ol{C(f,e)}$ is a section of the cone of nonnegative polynomials of degree $2d-2$: 
\begin{equation}\label{eq:RenCone}
 \overline{C(f,e)} \;\; = \;\; \{ a \in \R^{n} \;\;:\;\;  \Delta_{e,a}f \geq 0 \text{ on } \R^{n}\}.
 \end{equation}
\end{Thm}

\begin{proof}
  Let $C$ be the set on the right hand side of
  \eqref{eq:slicedInt}.
  From Theorem~\ref{thm:Interlacers}, we see that $D_a f$ interlaces
  $f$ with respect to $e$ for all $a\in C(f,e)$. This shows
  $C(f,e)\subset C$ and hence the inclusion $\overline{C(f,e)}\subset
  C$, since $C$ is closed. If this inclusion were strict, there would exist a point $a\in
  C\setminus\overline{C(f,e)}$ with $f(a)\neq 0$, since $C$ is also a
  convex cone by Corollary~\ref{cor:InterlacersNon}. 
  Thus to show the reverse inclusion, it therefore suffices to show that
  for any point $a$ outside of $\overline{C(f,e)}$ with $f(a)\neq 0$,
  the polynomial $D_a f$ does not belong to ${\rm Int}(f,e)$. If $a$
  belongs to $-\overline{C(f,e)}$, then $-D_a f$ belongs to ${\rm
    Int}(f,e)$. In particular, $-D_a f (e) >0$ and $D_a f$ does not
  belong to ${\rm Int}(f,e)$. Thus we may assume
  $a\notin\overline{C(f,e)}\cup -\overline{C(f,e)}$.
  Since $f$ is hyperbolic
  with respect to $e$, all of the roots $\alpha_1\leq \cdots \leq
  \alpha_d$ of $f(te+a)$ are real.  The reciprocals of these roots,
  $1/\alpha_1, \hdots 1/\alpha_d$, are roots of the polynomial
  $f(e+ta)$.

Because $a$ is not in $\overline{C(f,e)}\cup -\overline{C(f,e)}$, there is some $1\leq i <
n$ for which $\alpha_i < 0<\alpha_{i+1}$.  Since $f(e)\neq 0$ and $f(a)\neq 0$, we can take reciprocals to
find the roots of $f(e+ta)$:
\[ \frac{1}{\alpha_i} \leq \frac{1}{\alpha_{i-1}} \leq \cdots \leq \frac{1}{\alpha_1} < 0 < \frac{1}{\alpha_d} \leq \frac{1}{\alpha_{d-1}}\leq \cdots \leq \frac{1}{\alpha_{i+1}}.\]

By Rolle's Theorem, the roots of $\frac{\partial}{\partial t} f(e+ta)$ interlace those of $f(e+ta)$.  
Note that the polynomial $\frac{\partial}{\partial t} f(e+ta)$ is precisely $D_a f(e+ta)$, so the roots of $D_a f(e+ta)$ interlace those of $f(e+ta)$.
In particular, there is some root $\beta$ of $D_a f(e+ta)$ in the open interval $(1/\alpha_1, 1/\alpha_d)$, and thus $1/\beta \not\in [\alpha_1,  \alpha_d]$
is a zero of $D_a f(te+a)$. Therefore $D_a f(te+a)$ has only $d-2$ roots in the interval $[\alpha_1,  \alpha_d]$ and thus cannot interlace $f$ with respect to $e$.

Combining this with Theorem~\ref{thm:Interlacers} shows the equality in \eqref{eq:RenCone}.
\end{proof}

 \begin{figure}[h] 
 \includegraphics[width=5cm]{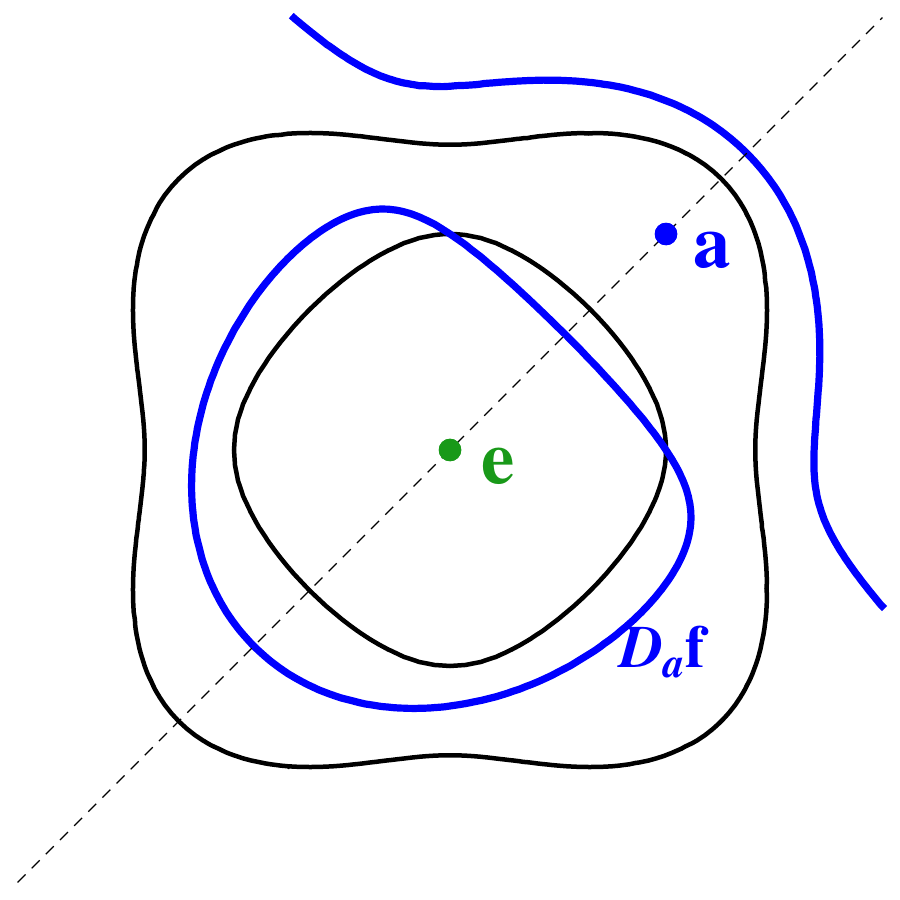}  \vspace{-.4cm}
\caption{For $a$ outside of the hyperbolicity cone, $D_af$ does not interlace $f$. }
\label{fig:cubic}
\end{figure}

\begin{Cor} Relaxing nonnegativity to sums-of-squares in \eqref{eq:RenCone} gives an inner approximation
to the hyperbolicity cone of $f$:
\begin{flalign}\label{eq:sosRelax}
 && \{a \in \R^{n}\;:\; \Delta_{e,a}f\ \text{ is a sum of squares } \} \;\; \subseteq \;\; \overline{C(f,e)}.&&\qed
 \end{flalign}
\end{Cor}

If the relaxation \eqref{eq:sosRelax} is exact, then the hyperbolicity
cone is a projection of a spectrahedron, namely of a section of the
cone of positive semidefinite matrices in ${\rm Sym}_N(\R)$,
where $N = \binom{n+d-2}{n-1} = \dim\R[x]_{d-1}$.  A polynomial $F$ is
a sum of squares if and only if there exists a positive semidefinite
matrix $G$ such that $F = v^TGv$, where $v$ is the vector of monomials
of degree at most $\deg(F)/2$. We call such a matrix $G$ a \emph{Gram
  matrix} of $F$.  The linear equations giving the Gram matrices of
$\Delta_{e,a}f$ give the desired section of ${\rm Sym}_N(\R)$.

If the relaxation \eqref{eq:sosRelax} is not exact, one can allow for
denominators in the sums of squares and successively improve the
relaxation. More precisely, for any integer $N\ge 0$ consider 
\begin{equation}
  \label{eq:sosRelaxDenom}
C_N=\bigl\{a \in \R^{n}\;:\; \bigl(\sum_{i=1}\nolimits^n\! x_i^2\bigr)^N\cdot \Delta_{e,a}f\ \text{ is a sum of squares } \bigr\} \;\; \subseteq \;\; \overline{C(f,e)}.
\end{equation}
As above, $C_N$ is seen to be a projection of a
spectrahedron. Furthermore, by a result of Reznick in \cite{Rez}, for
any positive definite form $F\in\R[x]$ there exists some positive
integer $N$ such that $(\sum_{i=1}^n x_i^2)^N\cdot F$ is a sum of
squares. Thus if $\sV_\R(f)$ is smooth, then $\{\Delta_{e,a}f\:|\:
a\in\R^n\}$ contains a strictly positive polynomial, for example $\Delta_{e,e}f$. 
It follows that the hyperbolicity cone $\overline{C(f,e)}$ is the closure of the
union of all the cones $C_N$. 

\begin{Remark}
  In a recent paper \cite{NS}, Netzer and Sanyal showed that the
  hyperbolicity cone of a hyperbolic polynomial without real
  singularities is the projection of a spectrahedron. Their proof uses
  general results on projected spectrahedra due to Helton and Nie and
  is not fully constructive. In particular, it does not imply anything
  about equality in \eqref{eq:sosRelax} or \eqref{eq:sosRelaxDenom}.

  Explicit representations of hyperbolicity cones as projected
  spectrahedra have recently been obtained by Parrilo and Saunderson in
  \cite{PS} for elementary symmetric polynomials and for directional
  derivatives of polynomials possessing a definite determinantal
  representation.
\end{Remark}

\begin{Remark} We also have the relaxation 
\begin{equation}\label{eq:sosmodIRelax}
 \{a \in \R^{n}\;:\; D_e f \cdot D_a f \ \text{ is a sum of squares modulo }( f ) \} \;\; \subseteq \;\; \overline{C(f,e)}.
 \end{equation}
It is unclear whether or not this relaxation is always 
equal to \eqref{eq:sosRelax}. Its exactness would also show $\overline{C(f,e)}$ to be the 
projection of a spectrahedron. We will see below that if $f$ has a definite determinantal representation, 
then we get equality in \eqref{eq:sosRelax} and \eqref{eq:sosmodIRelax}.
\end{Remark}

\begin{Example}
 Consider the quadratic form $f(x) = x_1^2 - x_2^2-\cdots-x_n^2$, which is hyperbolic
with respect to the point $e=(1,0,\hdots, 0)$.  The hyperbolicity cone $C(f,e)$ is known as the Lorentz cone. 
In this example, the relaxation \eqref{eq:sosRelax} is exact. To see this, note that
\begin{align*}
\Delta_{e,a}f \;\;&= \;\;(2x_1)(2a_1x_1-\sum_{j\neq1}2a_jx_j) - (x_1^2-\sum_{j\neq 1} x_j^2)(2a_1) \\
&=\;\;2(a_1x_1^2-2\sum_{j\neq 1} a_jx_1x_j +\sum_{j\neq 1}a_1x_j^2).
\end{align*}

Since every nonnegative quadratic form is a sum of squares, there is equality in \eqref{eq:sosRelax}.
In fact, taking the Gram matrix of $\frac{1}{2}\Delta_{e,a}f$, we recover the Lorentz cone as
\[ \overline{C(f,e)} \;\; =\;\;  \left\{a \in \R^{n}\;:\; 
\begin{pmatrix} 
a_1 	&-a_2	& \hdots 	& -a_n \\
-a_2 	&a_1	&		& 0	\\	
\vdots&		&\ddots	&\vdots\\
-a_n	& 0		&\hdots	&a_1
  \end{pmatrix} \succeq 0 \right\}.\]
  Note also that this Gram matrix gives a definite determinantal representation of $a_1^{n-2}f(a)$. 
\end{Example}

\begin{Example}\label{ex:cubic}
Consider the hyperbolic cubic polynomial 
\[f= (x - y) (x + y) (x + 2 y) - x z^2,\]
with $e=[1:0:0]$. Here the polynomial 
$\Delta_{e,a}f$ has degree four in $x,y,z$. 
In this case, the relaxation \eqref{eq:sosRelax} is exact, as shown in 
Corollary~\ref{cor:good3var}  below. 
(One can also see exactness from the fact that every nonnegative 
ternary quartic is a sum of squares). 
Using the basis $(x^2, y^2, z^2, x y, x z, y z)$ of $\R[x,y,z]_2$, we can then write
the cone $\ol{C(f,e)}$ as the set of $(a,b,c)$ in $\R^3$ for which
 there exists $(g_1, \hdots, g_6)\in \R^6$ to make the real symmetric matrix
\[ \begin{pmatrix}
 3 a + 2 b& g_1& g_2& 4 a -2 b& -2 c& g_3\\ 
 g_1& 9 a + 2 b& g_4& 4a - 8 b& g_5& -2 c\\ 
 g_2& g_4& a& g_6& 0& 0\\ 
4 a - 2 b& 4a - 8 b& g_6& 8 a - 20 b - 2 g_1& -2 c - g_3& -g_5\\
 -2 c& g_5& 0& -2 c - g_3& 2 b - 2 g_2& -2 a - g_6\\ 
 g_3& -2 c& 0& -g_5& -2 a - g_6& 2a + 6 b - 2 g_4
\end{pmatrix}\]
positive semidefinite. 
\end{Example}

The sums of squares relaxation  \eqref{eq:sosRelax} is not always exact. A counterexample comes from a 
multilinear hyperbolic polynomial and will be discussed in Example~\ref{ex:Vamos}.

\section{Definite Symmetric Determinants}\label{sec:HermDets}  
We consider $\det(X)$ as a polynomial in $\R[X_{ij}:i\leq j \in [d]]$, 
where $X = (X_{ij})$ is a symmetric matrix of variables.
Since all eigenvalues of a real symmetric matrix are real, this polynomial is 
hyperbolic with respect to the identity matrix.
The hyperbolicity cone $C(\det(X),I)$ is the cone of positive definite matrices. 
Hence, for any positive semidefinite matrix $E\neq 0$, the polynomial 
\begin{equation}\label{eq:RenMatrix}
D_E(\det(X)) \;\; = \;\; \trace\!\left(E\cdot X^{{\rm adj}}\right)
\end{equation}
interlaces $\det(X)$, where $X^{\rm adj}$ denotes the adjugate matrix,
whose entries are the signed $(d-1)\times (d-1)$-minors of $X$. This
holds true when we restrict to linear subspaces.  For real symmetric
$d\times d$ matrices $M_1, \hdots, M_n$ and variables $x=(x_1,\hdots,
x_n)$, denote
\[ M(\varx) \;=\; \sum_{j=1}^n x_j M_j.\]
If $M(e)$ is positive definite for some $e\in \R^{n}$, 
then the polynomial $ \det(M(x))$ is hyperbolic with respect to the point $e$. 

\begin{Prop}\label{prop:DetHyperbolic}
If $M$ is a real symmetric matrix of linear forms such that $M(e) \succ 0$ for some $e\in \R^{n}$, 
then for any positive semidefinite matrix $E$, the polynomial  $\trace\!\left(E\cdot M^{{\rm adj}}\right)$ 
interlaces $\det(M)$ with respect to $e$. 
\end{Prop}
\begin{proof}
By the discussion above, the polynomial $D_E(\det(X)) =\trace\!\left(E\cdot X^{{\rm adj}}\right)$ interlaces $\det(X)$ 
with respect to $E$. (In fact these are all of the interlacers of $\det(X)$. See Example~\ref{ex:SDP} below.)
By Theorem~\ref{thm:Interlacers}, $\trace\!\left(E\cdot X^{{\rm adj}}\right)$ interlaces $\det(X)$ with respect to 
any positive definite matrix, in particular $M(e)$. Restricting to the linear space $\{M(x) : x\in\R^{n}\}$ shows that 
$\trace\!\left(E\cdot M^{{\rm adj}}\right)$ interlaces $\det(M)$ with respect to $e$. 
\end{proof}

\begin{Thm} \label{thm:goodSOSrelax}
If $f\in \R[\varx]_d$ has a definite symmetric determinantal representation $f=\det(M)$ 
with $M(e)\succeq 0$ and $M(a)\succeq 0$, then 
$\Delta_{e,a}f$ is a sum of squares. 
In particular, there is equality in \eqref{eq:sosRelax}.
\end{Thm}

\begin{proof}
 Because $M(e)$ and $M(a)$ are positive semidefinite, we can
  write them as sums of rank-one matrices: $M(e) = \sum_i \lambda_i
  \lambda_i^T$ and $M(a) = \sum_j \mu_j \mu_j^T$, where $\lambda_i,
  \mu_j \in \R^d$.  Then $D_e f = \langle M(e), M^{\rm adj}
  \rangle=\langle \sum_i \lambda_i \lambda_i^T, M^{\rm adj} \rangle =
  \sum_i \lambda_i^T M^{\rm adj}\lambda_i$, so
\[ D_e f = \sum_i \lambda_i^T M^{\rm adj}\lambda_i \;\;\;\text{ and, similarly, }\;\;\; D_a f = \sum_j \mu_j^T M^{\rm adj}\mu_j.\] 
Furthermore, by Proposition~\ref{prop:HesseDet} below, the second derivative $D_aD_bf$ is
\[D_eD_af \; =D_e\left(\sum_j \mu_j^T M^{\rm adj}\mu_j\right) =  \; \sum_{i,j} u_{ij} \;\;\; \text{where} \;\;\;
u_{ij} = \left| 
\begin{matrix}
M & \lambda_i & \mu_j \\
\lambda_i^T & 0 & 0 \\
\mu_j^T & 0 & 0
\end{matrix}
  \right|.   \]
Now, again using Proposition~\ref{prop:HesseDet}, we see that $\Delta_{e,a}f $ equals
\begin{equation}\label{eq:HesseId}
\sum_{i,j} \biggl(( \lambda_i^T M^{\rm adj}\lambda_i )(\mu_j^T M^{\rm adj}\mu_j) - \det(M)\cdot  u_{ij}\biggl)
\;\;=\;\; \sum_{i,j} (\lambda_i^T M^{\rm adj}\mu_j)^2,
\end{equation}
which is the desired sum of squares.
\end{proof}

In fact, something stronger is true.  We can also consider the case where some power of $f$ has a definite determinantal representation. 
This is particularly interesting because taking powers of a hyperbolic polynomial does not change the hyperbolicity cone. 

\begin{Cor}\label{cor:power}
If $f\in\R[\varx]_d$ and a power $f^r$ has a definite symmetric determinantal representation $f^r=\det(M)$
with $M(e),M(a)\succeq 0$, then
$\Delta_{e,a}(f)$ is a sum of squares. 
In particular, there is equality in \eqref{eq:sosRelax}.
\end{Cor}

\begin{proof}
Let $f^r$ have a definite determinantal representation. 
 We have $\Delta_{e,a}(f^r)=r f^{2(r-1)} \Delta_{e,a}f.$
 Theorem \ref{thm:goodSOSrelax} states that $\Delta_{e,a}(f^r)$ is a sum of squares,
 \[g_1^2+\cdots+g_s^2 \;=\; r f^{2(r-1)} \Delta_{e,a}f\] for some $g_i \in \R[x]$.
 Let $p$ be an irreducible factor of $f^{2(r-1)}$. 
 Then $p$ is hyperbolic with respect to $e$ and the right hand side vanishes on $\mathcal{V}_{\C}(p)$. 
 Therefore, each $g_i$ vanishes on $\sV_\R(p)$ and thus on $\sV_\C(p)$, 
 since $\sV_\R(p)$ is Zariski dense in $\sV_\C(p)$.
 Thus we can divide the $g_i$ by $p$. By iterating this argument, we get the claim.
\end{proof}

\begin{Remark}
  This result is closely related to (but does not seem to follow from)
  \cite[Thm.~1.6]{NPT}, which says that the parametrized Hermite
  matrix of $f$ is a sum of matrix squares whenever a power of $f$
  possesses a definite determinantal representation.
\end{Remark}

\begin{Cor} \label{cor:good3var}
If $f\in \R[x_1, x_2, x_3]$, then there is equality in \eqref{eq:sosRelax}.
\end{Cor}
\begin{proof}
By the Helton-Vinnikov Theorem \cite{HV}, every hyperbolic polynomial in 
three variables has a definite determinantal representation. The claim then follows from Theorem~\ref{thm:goodSOSrelax}.
\end{proof}

\noindent The following
determinantal identities were needed in the proof of
Theorem~\ref{thm:goodSOSrelax} above.

\begin{Prop}\label{prop:HesseDet}
 Let $X$ be a $d\times d$ matrix of variables $X_{ij}$ and let $|\cdot|$ denote $\det(\cdot)$.
Then for any vectors $\alpha, \beta, \gamma, \delta \in \C^d$ we have
\begin{equation}\label{eq:bigHesse}
\left|\begin{matrix}X & \beta \\ \alpha^T & 0 \end{matrix}\right | \cdot \left|\begin{matrix}X & \delta \\ \gamma^T & 0 \end{matrix}\right |
-\left|\begin{matrix}X & \delta \\ \alpha^T & 0 \end{matrix}\right | \cdot \left|\begin{matrix}X & \beta \\ \gamma^T & 0 \end{matrix}\right |
\;\;=\;\; |X| \cdot \left|\begin{matrix}X & \beta &\delta  \\ \alpha^T & 0 & 0 \\ \gamma^T &0 & 0 \end{matrix}\right |
\end{equation}
in $\C[X_{ij}:\; 1\leq i,j\leq d]$. Furthermore, 
 \[ D_{\beta \alpha^T}|X| = \left|\begin{matrix}X & \beta \\ \alpha^T & 0 \end{matrix}\right | \;\;\;\;\; \text{and}\;\;\;\;\;
D_{\delta\gamma^T } D_{ \beta \alpha^T }|X| = \left|\begin{matrix}X & \beta &\delta  \\ \alpha^T & 0 & 0 \\ \gamma^T &0 & 0 \end{matrix}\right |.
 \]
\end{Prop}
\begin{proof}
We will prove the first identity using Schur complements. See, for example, \cite[\S 1]{BS}.
If $A$ is a $m\times m$ submatrix of the $n\times n$ matrix $\begin{pmatrix}A & C \\ B & D \end{pmatrix}$, 
then its determinant equals $|A| \cdot |D-BA^{-1}C|$.
If $D$ is the zero matrix, this simplifies to
\[
\left|\begin{matrix}A & C \\ B & 0 \end{matrix}\right| 
  =  |A|\cdot \left| \frac{-1}{|A|}\cdot BA^{\rm adj}C  \right|
  =  |A| \cdot \left(\frac{-1}{|A|}\right)^{n-m} \cdot  | BA^{\rm adj}C|.
\]
To obtain the desired identity, we set  
$A=X$, $B = \begin{pmatrix} \alpha^T \\ \gamma^T \end{pmatrix}$, and $C = \begin{pmatrix} \beta & \delta \end{pmatrix}$:
\[
\left|\begin{matrix}X & \beta &\delta  \\ \alpha^T & 0 & 0 \\ \gamma^T &0 & 0 \end{matrix} \right|
 = |X| \cdot \left(\frac{-1}{|X|}\right)^{2}  \cdot \left|\begin{pmatrix} \alpha^T \\ \gamma^T \end{pmatrix}  X^{\rm adj}  \begin{pmatrix} \beta & \delta \end{pmatrix}\right|\\
 = \frac{1}{|X|} \cdot\left|\begin{matrix} \alpha^T X^{\rm adj} \beta & \alpha^T X^{\rm adj} \delta \\ \gamma^T X^{\rm adj} \beta & \gamma^T X^{\rm adj} \delta  \end{matrix}\right|.
\]
Multiplying both sides by $\det(X)$ finishes the proof of the determinantal identity. 

For the claim about derivatives of the determinant, by additivity, we only need to look at the case when $\alpha, \beta, \gamma, \delta$ are unit vectors, 
$e_i, e_j, e_k, e_l$, respectively. 
Then $D_{\beta \alpha^T}|X| = D_{e_j e_i^T}|X|$ is the derivative of $|X|$ with respect to the entry $X_{ji}$. This is the signed minor 
of $X$ obtained by removing the $j$th row and $i$th column, which is precisely the determinant $\left|\begin{matrix}X & e_j \\ e_i^T & 0 \end{matrix}\right |$.
Taking the derivative of this determinant with respect to $X_{lk}$ the same way gives 
\[\frac{\partial^2 |X|}{\partial X_{ji} \partial X_{lk}} \;=\;
D_{ e_l e_k^T} D_{e_je_i^T }|X| \;=\; 
D_{ e_le_k^T} \left|\begin{matrix}X & e_j \\ e_i^T & 0 \end{matrix}\right | \; = \;
\left|\begin{matrix}X & e_j &e_l  \\ e_i^T & 0 & 0 \\ e_k^T &0 &
    0 \end{matrix}\right |.\qedhere
\]
\end{proof}

We conclude this section with a general result inspired by Dixon's
construction of determinantal representations of plane curves, which
will be applied in the next section. If $f\in \R[x]_d$ has a definite determinantal representation, 
$f=\det(M)$ with $M(e)\succ 0$, then $M^{\rm adj}$ is a $d \times d$
matrix with entries of degree $d-1$. This matrix has rank at most one on $\sV(f)$, as seen by the identity 
$M\cdot M^{\rm adj} = \det(M)\cdot I$. By Proposition~\ref{prop:DetHyperbolic}, 
the top left entry $M^{\rm adj}_{11}$ interlaces $f$ with respect to $e$. In fact, these properties
of $M^{\rm adj}$ are enough to reconstruct a definite determinantal representation $M$. 

\begin{Thm}\label{thm:NonVanishingDet} Let $A=(a_{ij})$ be a symmetric $d \times d$ matrix of real forms of degree $d-1$. 
Suppose that $f\in \R[\varx]_d$ is irreducible and hyperbolic 
with respect to $e \in \R^n$. If $A$ has rank one modulo $(f)$, then $f^{d-2}$ divides 
the entries of $A^{\rm adj}$ and the matrix $M=(1/f^{d-2})A^{\rm
    adj}$ has linear entries. Furthermore there exists $\gamma\in\R$ such that
\[
\det(M) = \gamma f.
\]
If $a_{11}$ interlaces $f$ with respect to $e$ and $A$ has full rank, then 
$\gamma \neq 0$ and $M(e)$ is definite. 
\end{Thm}

\begin{proof}
  By assumption, $f$ divides all the $2\times 2$ minors of
  $A$. Therefore, $f^{d-2}$ divides all of the $(d-1)\times(d-1)$
  minors of $A$ and thus all of the entries of the adjugate matrix
  $A^{\rm adj}$, see \cite[Lemma 4.7]{PV}.  We can then consider the
  matrix $M = (1/f^{d-2})\cdot A^{\rm adj}$.  By similar arguments,
  $f^{d-1}$ divides $\det(A)$. Because these both have degree
  $d(d-1)$, we conclude that $\det(A) = \lambda f^{d-1}$ for some
  $\lambda \in \R$.  Putting all of this together, we find that
 \[\det(M) \;\;=\;\;\frac{1}{f^{d(d-2)}}\cdot \det(A^{\rm adj})\;\;=\;\;\frac{1}{f^{d(d-2)}} \det(A)^{d-1}\;\;=\;\;\lambda^{d-1} f,\]
 so we can take $\gamma=\lambda^{d-1}$.  Now, 
 suppose that $a_{11}$ interlaces $f$ and that 
 $\gamma=\lambda =0$. Then $\det(A)$ is identically
 zero. In particular, the determinant of $A(e)$ is zero, there
 is some nonzero vector $v\in \R^d$ in its kernel, and $v^TA(e)v$ is
 also zero.

We will show that the polynomial $v^TAv$ is not identically zero and that it 
interlaces $f$ with respect to $e$.  This will contradict the conclusion that $v^TA(e)v=0$.
Because $A$ has rank one on $\sV(f)$, for each $i=1,\hdots, d$ we have that 
\begin{equation}\label{eq:2x2minors}
 (e_i^TAe_i)(v^TAv) - (e_i^TAv)^2  = 0 \;\; \text{ modulo }\;(f).
 \end{equation}
If $v^TAv$ is identically zero, then $e_i^TAv$ vanishes on $\sV(f)$. 
Since $e_i^TAv$ only has degree $d-1$, it must vanish identically as well. 
As this holds for each $i$, this implies that $Av$ is zero, which contradicts our assumption. 
Thus $v^TAv$ cannot be identically zero. 

Furthermore, \eqref{eq:2x2minors} shows that $a_{11}\cdot (v^TAv)$ is nonnegative on $\sV_{\R}(f)$. 
Then Lemma~\ref{lem:nonnegInterlacers} shows that $v^TAv$ interlaces $f$ with respect to $e$. 
In particular, $v^TAv$ cannot vanish at the point $e$. Thus the determinant of $A$ and 
hence $M$ cannot be identically zero.

Thus $M$ is a determinantal representation of $f$. To show that $M(e)$ is definite, 
it suffices to show that $A(e)$ is definite. For any vector $v \in \R^d$, we see from 
\eqref{eq:2x2minors} with $i=1$ that $a_{11}v^TAv$ is nonnegative on $\sV_{\R}(f)$.
Thus $a_{11}(e) \cdot v^TAv$ belongs to ${\rm Int}(f,e)$ by Lemma~\ref{lem:nonnegInterlacers}
and in particular $a_{11}(e) \cdot v^TA(e)v$ is positive for all $v\in \R^d$. Hence the matrix $A(e)$ is definite. 
\end{proof}

\section{Multiaffine polynomials}\label{sec:multiaff}
An interesting special case of a hyperbolic polynomial is a
multiaffine polynomial whose hyperbolicity cone contains the positive
orthant. These polynomials are deeply connected to the theory of
matroids \cite{Bra11, COSW, WW}.

\begin{Def}
  A polynomial $f\in \R[\varx]$ is called \textbf{affine} in $x_i$ if
  the degree of $f$ in $x_i$ is at most one.  If $f$ is affine in each
  variable $x_1, \ldots, x_n$, then $f$ is called
  \textbf{multiaffine}.
\end{Def} 

Much of the literature on these polynomials deals with complex polynomials, rather than real polynomials, 
and the property of \emph{stability} in place of hyperbolicity. 

\begin{Def}
 A polynomial $f\in \C[\varx]$ is called \textbf{stable} if $f(\mu)$ is non-zero 
 whenever the imaginary part of each coordinate $\mu_i$ is positive for all $1 \leq i \leq n$.
\end{Def}

A real homogeneous polynomial $f \in \R[\varx]$ is stable if and only
if $f$ is hyperbolic with respect to every point in the positive orthant.
After a linear change of variables, every hyperbolic polynomial
is stable.  In 2004, Choe, Oxley, Sokal, and Wagner \cite{COSW} showed
that if $f\in \R[\varx]_d$ is stable, homogeneous, and multiaffine,
then its \emph{support} (the collection of $I\subset \{1,\dots,n\}$ for which
the monomial $\prod_{i\in I}x_i$ appears in $f$) is the set of bases
of a matroid. They further show that any \emph{representable} matroid
is the support of some stable multiaffine polynomial. In 2010,
Br\"and\'en \cite{Bra11} used this deep connection to disprove the
generalized Lax conjecture by showing that the bases-generating
polynomial of the V\'amos matroid (see Example~\ref{ex:Vamos}) is
hyperbolic but none of its powers has a determinantal representation.
This example will also provide a counterexample to equality in our
relaxation \eqref{eq:sosRelax}.

The Wronskian polynomials $\Delta_{e,a}f$ also played a large part in the 
study of multiaffine stable polynomials. They are particularly useful when the points $e$ and $a$ are 
unit vectors. In this case, we will simplify our notation and write
\[ \Delta_{ij}(f) \;\;:=\;\;\Delta_{e_i, e_j}(f) \;\;=\;\; \frac{\partial f}{\partial x_i} \cdot \frac{\partial f}{\partial x_j} - f\cdot \frac{\partial^2 f}{\partial x_i \partial x_j}. \]
Using these polynomials, Br\"and\'en \cite{Bra07}
established  a necessary and sufficient condition for multiaffine polynomials to be stable.

\begin{Thm}[\cite{Bra07}, Theorem~5.6]{\label{thm:brand}}
For multiaffine $f\in \R[\varx]$, the following are equivalent:
 \begin{enumerate}
  \item $\Delta_{ij} f$ is 
  nonnegative on $\R^n$ for all $1 \leq i,j \leq n$,
  \item $f$ is stable.
 \end{enumerate}
\end{Thm}

Br\"and\'en also notes that the implication (2)$\Rightarrow$(1)
holds for polynomials that are not multiaffine, giving an alternative
proof of a part of Theorem~\ref{thm:Interlacers} above. The other
implication however, does not hold in general, as the following
example shows.

\begin{Example}
Let $h=x_1^2+x_2^2$, 
  $q=x_1+x_2$ and $N \in \N$.
 Clearly $q^N h$ is not hyperbolic with respect to any $e \in \R^2$,
 but for all $i,j \in \{1,2\}$ we have
 \begin{align*}
  \Delta_{ij}(q^N h)\:&=\:q^{2N} \Delta_{ij}h + N q^{2N-2} h^2 \Delta_{ij}q \\
  &=\:q^{2N-2}(q^2 \Delta_{ij}h+ N h^2).
 \end{align*}
 Now let $z \in \R$ be the minimal value that $q^2 \Delta_{ij}h$
 takes on the unit sphere and let $N>|z|$. Then, since $\Delta_{ij}(q^N h)$ is homogeneous, we see that $\Delta_{ij}(q^N h)$
 is nonnegative on $\R^2$.  Because $\Delta_{ij}(q^N h)$ is a
 homogeneous polynomial in two variables, it is even a sum of squares.
 Thus $\Delta_{ab} (q^N h)$ is a sum of squares for all $a,b$ in the
 positive orthant. 
 This also shows that
 the converse of Corollary \ref{cor:power} is not true, i.e. there is some polynomial $f$ such that $\Delta_{e,a}f$ is a sum of squares for all $e,a$ in some
 full dimensional cone, but no power of $f$ has a definite
 determinantal representation. 
\end{Example}

In an analogous statement, the polynomials $\Delta_{ij}$ can also be used to determine 
whether or not a homogeneous multiaffine stable polynomial has a definite determinantal representation.

\begin{Thm}{\label{thm:muaff}}
 Let $f\in \R[\varx]_d$ be homogeneous and stable. Suppose $f$ is affine in the variables $x_1, \ldots, x_d$
 and the coefficient of $x_1 \cdots x_d$ in $f$ is non-zero.
Then the following are equivalent:
 \begin{enumerate}
  \item $\Delta_{ij} f$ is 
  a square in $\R[\varx]$ for all $1 \leq i,j \leq d$;
 \item $\frac{\partial f}{\partial x_i} \cdot \frac{\partial f}{\partial x_j}$ is a square in $\R[\varx]/(f)$ for all $1 \leq i,j \leq d$;
  \item $f$ has a definite determinantal representation.
 \end{enumerate}
\end{Thm}

\begin{Lemma}{\label{lem:sqfac}}
 Let $f\in \R[\varx]$ be affine in $x_i$ and $x_j$ for some $i,j \in\{1,\dots,n\}$. If $f=g\cdot h$ with $g,h \in \R[\varx]$, then $\Delta_{ij} f$ is a square if and only if
 $\Delta_{ij} g$ and $\Delta_{ij} h$ are squares.
\end{Lemma}

\begin{proof}
Suppose $\Delta_{ij} f$ is a square. Since $f$ is affine in $x_i, x_j$, both $g$ and $h$ are affine in $x_i, x_j$
and either $\frac{\partial g}{\partial x_i}=0$ or $\frac{\partial h}{\partial x_i}=0$.
It follows that either $\Delta_{ij}g=0$ or $\Delta_{ij}h=0$.
Using the identity $\Delta_{ij} f = g^2 \Delta_{ij}h + h^2 \Delta_{ij} g$, we see that 
either $\Delta_{ij}g=0$ or $\Delta_{ij}g=(\Delta_{ij} f )/ h^2.$
In both cases $\Delta_{ij} g$ is a square. The same holds true for $\Delta_{ij} h$.
For the converse, suppose that $\Delta_{ij} g$ and $\Delta_{ij} h$ are squares. 
As we saw above, one of them is zero. Thus $\Delta_{ij} f=h^2 \Delta_{ij} g$ or
$\Delta_{ij} f=g^2 \Delta_{ij} h$.
\end{proof}

\begin{proof}[Proof of Thm.~\ref{thm:muaff}] 
 ($1 \Rightarrow 2$) 
Clear. \smallskip

($2\Rightarrow 3$)
For a start, suppose that $f$ is irreducible. We will construct a
matrix $A$ satisfying the hypotheses of Theorem~\ref{thm:NonVanishingDet}.
For every $i \leq j$, the polynomial $\frac{\partial f}{\partial x_i} \cdot \frac{\partial f}{\partial x_j}$ is 
equivalent to a square $a_{ij}^2$ modulo $(f)$. In the case $i=j$ we can choose $a_{ii}=\frac{\partial f}{\partial x_i}$.
Then it is easy to check that $a_{11}a_{ii}$ equals $a_{1i}^2$ modulo $(f)$.
Further, for every $2\leq i<j \leq d$, the polynomials $(a_{11}a_{ij})^2$ and $(a_{1i}a_{1j})^2$ are equivalent modulo $f$. 
After changing the sign of $a_{ij}$ if necessary, we see that $a_{11}a_{ij}$ equals $a_{1i}a_{1j}$ modulo $(f)$. 
Because $f$ is irreducible, it follows that the symmetric matrix $A =(a_{ij})_{ij}$ has rank one on $\sV(f)$.

We now need to show that $A$ has full rank.  For each $k=1, \hdots, d$, consider the point
$p_k=\sum_{j\in[d]\backslash\{k\}}e_j$, which lies in the real variety
of $f$.  For $j\neq k$, we see that $\partial f / \partial x_j$
vanishes at $p_k$, and therefore so must $a_{kj}$. On the other hand,
$a_{kk}(p_k) = \partial f / \partial x_k (p_k)$ equals the nonzero
coefficient of $x_1\cdots x_d$ in $f$. Now suppose that $Av=0$ for
some $v\in \R^d$. The $k$th row of this is $\sum_j v_j a_{kj}=0$.
Plugging in the point $p_k$ then shows that $v_k$ must be zero, and
thus $v$ is the zero vector.  Since $f$ is stable, $a_{11}=\partial f/\partial x_1$ interlaces it,
and so by  Theorem~\ref{thm:NonVanishingDet}, $f$ has a definite determinantal representation.  

If $f$ is reducible and $g$ is an irreducible factor of
$f$, then, by Lemma \ref{lem:sqfac}, $\Delta_{ij} g$ is a square. Since
every irreducible factor of $f$ has a definite determinantal
representation,  so has $f$.

 ($3 \Rightarrow 1$)
 Let $f = \det(M)= \det(\sum_ix_i M_i)$ where $M_1, \hdots, M_n$ are real symmetric $d\times d$ matrices where 
 $\sum_i M_i \succ 0$. Because $f$ is affine in each of the variables $x_1, \hdots, x_d$, the matrices 
$M_1, \hdots, M_d$ must have rank one. Furthermore, since $f$ is stable, these rank-one matrices must be positive 
semidefinite (see \cite{Bra11}, proof of Theorem~2.2). Thus we can write $M_i = v_i v_i^{ T}$, with $v_i \in \R^{d}$ for each $1 \leq i \leq d$. 
Then by \eqref{eq:HesseId} and Proposition \ref{prop:HesseDet}, we have  $\Delta_{ij} f = (v_i^T M^{\rm adj}v_j)^2$ for $1 \leq i,j \leq d$. \smallskip
\end{proof}

\begin{Cor}
 Let $f\in \R[\varx]$ be homogeneous, stable and multiaffine.
 Then the following are equivalent:
 \begin{enumerate}
  \item $\Delta_{ij} f$ is 
  a square for all $1 \leq i,j \leq n$;
    \item $\frac{\partial f}{\partial x_i} \cdot \frac{\partial f}{\partial x_j}$ is a square in $\R[\varx]/(f)$ for all $1 \leq i,j \leq n$;
  \item $f$ has a definite determinantal representation.
 \end{enumerate}
\end{Cor}

\begin{proof}
 This is an immediate consequence of the preceding theorem.
\end{proof}

\begin{Cor} 
 Let $1 \leq k \leq n$ and let $f\in \R[\varx]$ be a multiaffine stable polynomial.
 If $f$ has a definite determinantal representation, then $\frac{\partial f}{\partial x_k}$ and $f|_{x_k=0}$
 also have a definite determinantal representation.
\end{Cor}

\begin{proof}
Let $1 \leq k,i,j \leq n$, $g=\frac{\partial f}{\partial x_k}$ and $h=f|_{x_k=0}$. 
Wagner and Wei \cite{WW} calculated 
\[ \Delta_{ij} f \;\;=\;\; x_k^2 \cdot \Delta_{ij} g + x_k \cdot p+\Delta_{ij} h, \] 
where $p, g, h \in \R[x_1, \ldots , x_n]$ do not depend on $x_k$.
Since $\Delta_{ij} f$ is a square, $\Delta_{ij} g$ and $\Delta_{ij} h$ are squares as well. 
Thus $g$ and $h$ have a definite determinantal representation.
\end{proof}

\begin{Cor}
 Let $f=g\cdot h$, where $f,g,h \in \R[\varx]$ are multiaffine stable polynomials. 
 Then $f$ has a definite determinantal representation if and only if both
 $g$ and $h$ do.
\end{Cor}

\begin{proof}
 This follows directly from Lemma \ref{lem:sqfac} and Theorem \ref{thm:muaff}.
\end{proof}

\begin{Example}[Elementary Symmetric Polynomials]
 Let $e_d \in \R[x]$ be the elementary symmetric polynomial of degree $d$. We have $\Delta_{ij} e_1=1$, $\Delta_{ij} e_n=0$ and
 $\Delta_{ij} e_{n-1}= (x_1 \ldots x_n/x_i x_j)^2$ for all $1 \leq i<j \leq n$. It is a classical result
 that these are the only cases
 where $e_d$ has a definite determinantal representation \cite{San}. 
 Indeed, for $n\geq 4$ and $2 \leq d \leq n-2$ the coefficients of the monomials $(x_3 x_5 \cdots x_{d+2})^2$, 
 $(x_4 x_5 \cdots x_{d+2})^2$ and $x_3 x_4 ( x_5 \cdots x_{d+2})^2$ in $\Delta_{12} e_d$ are all $1$.
 Specializing to $x_j=1$ for $j\geq 5$ then shows that  $\Delta_{12}f$ is not a square.
\end{Example} 

\begin{Example}[The V\'amos Polynomial]
\label{ex:Vamos}
The relaxations \eqref{eq:sosRelax} and \eqref{eq:sosmodIRelax} are not always exact. 
An example of this comes from the multiaffine quartic polynomial in $\R[x_1, \hdots, x_8]_4$ 
given as the bases-generating polynomial of the V\'amos matroid: 
\[h(x_1, \hdots, x_8) = \sum_{I \subset \binom{[8]}{4} \backslash C} \prod_{i\in I} x_i, \]
where $C = \{\{1,2,3,4\},\{1,2,5,6\},\{1,2,7,8\},\{3,4,5,6\},\{3,4,7,8\}\}$.
Wagner and Wei \cite{WW} have shown that the polynomial $h$ is stable, using an improved version of Theorem \ref{thm:brand} and
representing $\Delta_{13}h$ as a sum of squares.
But it turns out
that $\Delta_{78} h$ is not a sum of squares. Because the cone of sums
of squares is closed, it follows that for some $a,e$ in the
hyperbolicity cone of $h$, the polynomial 
 $D_{e}h\cdot D_{a}h - h\cdot D_{e}D_{a}h$ is not a sum of squares. 
 In order to show that $\Delta_{78} h$ is not a sum of squares, it suffices to 
 restrict to the subspace $\{x =x_1 = x_2,\;y =x_3=x_4, \;z =x_5=x_6\}$ and show that the resulting polynomial 
$W=(1/4)\Delta_{78} h (x,x,y,y,z,z,w,w)$ is not a sum of squares. This restriction is given by 
\begin{align*}
W \;=\; &x^4y^2 + 2x^3y^3 + x^2y^4 + x^4yz + 5x^3y^2z + 
6x^2y^3z + 2xy^4z + x^4z^2 + 5x^3yz^2 + 10x^2y^2z^2 \\
&+ 6xy^3z^2 + y^4z^2 + 2x^3z^3 + 6x^2yz^3 + 6xy^2z^3 + 
2y^3z^3 + x^2z^4 + 2xyz^4 + y^2z^4.
\end{align*}
This polynomial vanishes at six points in $\P^2(\R)$, 
\[[1:0:0], \;[0:1:0],\; [0:0:1],\;[1:-1:0],\;[1:0:-1], \text{ and }\;[0:1:-1].\]
Thus if $W$ is written as a sum of squares $\sum_k h_k^2$, then each $h_k$ must 
vanish at each of these six points. The subspace of $\R[x,y,z]_3$ of cubics vanishing
in these six points is four dimensional and spanned by 
$v =\{x^2y+xy^2, x^2z+xz^2, y^2z+yz^2, xyz\}$. 
Then $W$ is a sum of squares if and only if there exists a positive semidefinite 
$4\times 4$ matrix $G$ such that $W = v^TGv$.  However, the resulting 
linear equations in the variables $G_{ij}, \; 1\leq i\leq j \leq 4$,
have the unique solution
\[G = 
\begin{pmatrix}
1& 1/2&1& 2\\
1/2& 1& 1& 2\\ 
1& 1& 1& 2\\
2& 2& 2& 5
\end{pmatrix}.
 \]
One can see that $G$ is not positive semidefinite from its determinant, which is $-1/4$. 
Thus $W$ cannot be written as a sum of squares. 

This, along with Corollary~\ref{cor:power}, provides another proof that no power of the V\'amos polynomial $h(x)$
has a definite determinantal representation. 

The polynomial $\Delta_{78} h$ is also not a sum of squares modulo the
ideal $(h)$.  To see this, suppose $\Delta_{78} h=\sum_i q_i^2 +
p\cdot h$ for some $p,q_i \in \R[x]$ and consider the terms with
largest degree in $x_7$ and $x_8$ in this expression.  Writing 
$h= h_0+h_1(x_7+x_8)+x_7x_8h_2$ where $h_0, h_1, h_2$ lie in $\R[x_1,
\hdots, x_6]$, we see that the leading form $x_7x_8 h_2$ is
\textit{real radical}, meaning that whenever a sum of squares $\sum_i
g_i^2$ lies in the ideal $(x_7x_8 h_2)$, this ideal contains each
polynomial $g_i$. Since $\Delta_{78} h$ does not involve the variables
$x_7$ and $x_8$, we can then reduce the polynomials $q_i$ modulo the
ideal $(h)$ so that they do not contain the variables $x_7, x_8$. See
\cite[Lemma 3.4]{realrad}.  Because $h$ does involve the variable
$x_7$ and $x_8$, this results in a representation of $\Delta_{78} h$
as a sum of squares, which is impossible, as we have just seen.
\end{Example}

\section{The cone of interlacers and its boundary} \label{sec:boundary}

Here we investigate the convex cone ${\rm Int}(f,e)$ of polynomials interlacing $f$. 
We compute this cone in two examples coming from optimization 
and discuss its \emph{algebraic boundary}, the minimal polynomial 
vanishing on the boundary of ${\rm Int}(f,e)$, when this cone is full dimensional. 
For smooth polynomials, this algebraic boundary is irreducible. 

If the real variety of a hyperbolic polynomial $f$ is smooth, 
then the cone ${\rm Int}(f,e)$ of interlacers is full dimensional in $\R[\varx]_{d-1}$. 
On the other hand, if $\sV(f)$ has a real singular point, then every polynomial 
that interlaces $f$ must pass through this point. This has two interesting 
consequences for the hyperbolic polynomials coming from linear programming 
and semidefinite programming. 

\begin{Example}\label{ex:LP}
Consider $f = \prod_{i=1}^nx_i$. The singular locus of $\sV(f)$ consists of the set of vectors 
with two or more zero coordinates. The subspace of polynomials in $\R[x]_{n-1}$ 
vanishing in these points is spanned by the $n$ polynomials 
$\{\prod_{j\neq i}x_j\::\; i=1,\hdots, n\}$. Note that this is exactly the 
linear space spanned by the partial derivatives of $f$. Theorem~\ref{thm:niceCones}
then shows that the cone of interlacers is isomorphic to $\overline{C(f,e)} = (\R_{\geq 0})^{n}$:
\[{\rm Int}\left( \prod x_i, \mathbf{1}\right) 
\;\;=\;\; \biggl\{\; \sum_{i=1}^n a_i \prod_{j\neq i}x_j \;:\; a\in (\R_{\geq 0})^{n} \;\biggl\} 
\;\; \isom \;\; (\R_{\geq 0})^{n}. \]
\end{Example}

Interestingly, this also happens when we replace the positive orthant by the cone of positive definite symmetric matrices. 

\begin{Example}\label{ex:SDP}
Let $f = \det(X)$ where $X$ is a 
$d\times d$ symmetric matrix of variables.  The singular locus of $\sV(f)$ is the locus of 
matrices with rank $\leq~d-2$. The corresponding ideal 
is generated by the $(d-1)\times(d-1)$ minors of $X$. 
Since these have degree $d-1$, we see that the polynomials 
interlacing $\det(X)$ must lie in the linear span of the
$(d-1)\times(d-1)$ minors of $X$. 
Again, this is exactly the linear span of the directional derivatives $D_E(f) = \trace(E\cdot X^{\rm adj})$. 
Thus Theorem~\ref{thm:niceCones} identifies ${\rm Int}(f,e)$ with the cone of positive semidefinite 
matrices:
\[{\rm Int}\left( \det(X), I\right) 
\;\;=\;\; \biggl\{\; \trace(A\cdot X^{\rm adj})\;:\; A\in \R^{d\times d}_{\succeq 0} \;\biggl\} 
\;\; \isom \;\; \R^{d\times d}_{\succeq 0}. \]
\end{Example}

 If $\sV_{\R}(f)$ is nonsingular, then the cone ${\rm Int}(f,e)$ is full dimensional and its algebraic boundary is a 
 hypersurface in $\R[\varx]_{d-1}$.  We see that any polynomial $g$ on the boundary 
 of ${\rm Int}(f,e)$ must have a non-transverse intersection point with $f$. As we see in the next theorem, 
 this algebraic condition exactly characterizes the algebraic boundary of ${\rm Int}(f,e)$.

 \begin{Thm}
   Let $f \in \R[\varx]_d$ be hyperbolic with respect to $e\in \R^n$ and
   assume that the projective variety $\sV(f)$ is smooth. Then the
   algebraic boundary of the convex cone ${\rm Int}(f,e)$ is the
   irreducible hypersurface in $\C[x]_{d-1}$ given by
\begin{equation}\label{eq:algBound}
  \left\{g\in \C[\varx]_{d-1} \;:\; \exists \; p\in\P^{n-1} \text{ such that } f(p)=g(p)=0 
 \text{ and }\rank\begin{pmatrix} \nabla f (p) \\ \nabla g (p) \end{pmatrix}  \leq 1\right\}.
 \end{equation}
 \end{Thm}

\begin{proof}
First, we show that the set \eqref{eq:algBound} is irreducible. Consider the incidence variety $X$
of polynomials $g$ and points $p$ satisfying this condition,
\[X \;=\; \left\{(g,p) \in \P(\C[\varx]_{d-1})\times\P^{n-1}
  \;:\;f(p)=g(p)=0 \text{ and }\rank\begin{pmatrix} \nabla f (p) \\
    \nabla g (p) \end{pmatrix} \leq 1\right\}.\] The projection
$\pi_2$ onto the second factor is $\sV(f)$ in $\P^{n-1}$.  Note that
the fibres of $\pi_2$ are linear spaces in $\P(\C[x]_{d-1})$ of constant
dimension. In particular, all fibres of $\pi_2$ are irreducible of the
same dimension. Since $X$ and $\sV(f)$ are projective and the latter
is irreducible, this implies that $X$ is irreducible (see \cite[\S
I.6, Thm.~8]{Sha}), so its projection $\pi_1(X)$ onto the first
factor, which is our desired set \eqref{eq:algBound}, is also
irreducible.

If $\sV(f)$ is smooth, then by \cite[Lemma~2.4]{PV}, $f$ and $D_ef$ share no real roots. 
This shows that the set of polynomials $g\in \R[x]_{d-1}$ for which $D_ef\cdot g$ is 
strictly positive on  $\sV_{\R}(f)$ is nonempty, as it contains $D_ef$ itself. This set is open 
and contained in ${\rm Int}(f,e)$, so ${\rm Int}(f,e)$ is full dimensional 
in $\R[\varx]_{d-1}$. Thus its algebraic boundary $\ol{\partial{\rm Int}(f,e)}$ 
is a hypersurface in $\C[x]_{d-1}$. To finish the proof, we just need to show that 
this hypersurface is contained in \eqref{eq:algBound}, since the latter is irreducible. 

To see this, suppose that $g\in \R[\varx]_{d-1}$ lies in the boundary
of ${\rm Int}(f,e)$.
By Theorem~\ref{thm:Interlacers}, there is some point $p\in \sV_{\R}(f)$ at which 
$g\cdot D_ef$ is zero. As $f$ is nonsingular, $D_ef(p)$ cannot be zero, again using 
\cite[Lemma~2.4]{PV}. Thus $g(p)=0$. Moreover, the polynomial 
$g\cdot D_ef - f\cdot D_eg$ is globally nonnegative, so its gradient also 
vanishes at the point $p$. As $f(p)=g(p)=0$, this means that 
$D_ef(p)\cdot \nabla g (p) = D_eg(p)\cdot \nabla f(p)$. Thus the pair 
$(g,p)$ belongs to $X$ above. 
\end{proof}

$\;$

When $\sV(f)$ has real singularities, computing the dimension of ${\rm Int}(f,e)$ becomes 
more subtle. In particular, it depends on the type of singularity.
\begin{Example}
Consider the two hyperbolic quartic polynomials 
\[f_1 = 3 y^4 + x^4 + 5 x^3 z + 6 x^2 z^2 - 6 y^2 z^2 \;\;\text{  and  } \;\;
f_2 = (x^2 + y^2 + 2 x z)(x^2 + y^2 + 3 x z),\] whose real varieties are shown in 
Figure~\ref{fig:singularQuartics} in the plane $\{z=1\}$. Both are hyperbolic with respect to the
point $e=[-1:0:1]$ and singular at $[0:0:1]$. Every polynomial interlacing either of these
must pass through the point $[0:0:1]$. However, for a polynomial $g$ to interlace $f_2$, its partial derivative
 $\partial g/\partial y$ must also vanish at $[0:0:1]$. Thus ${\rm Int}(f_1,e)$ has codimension one in 
 $\R[x,y,z]_3$ whereas ${\rm Int}(f_2,e)$ has codimension two. 
   \begin{figure}[h] 
   \includegraphics[height=3.6cm]{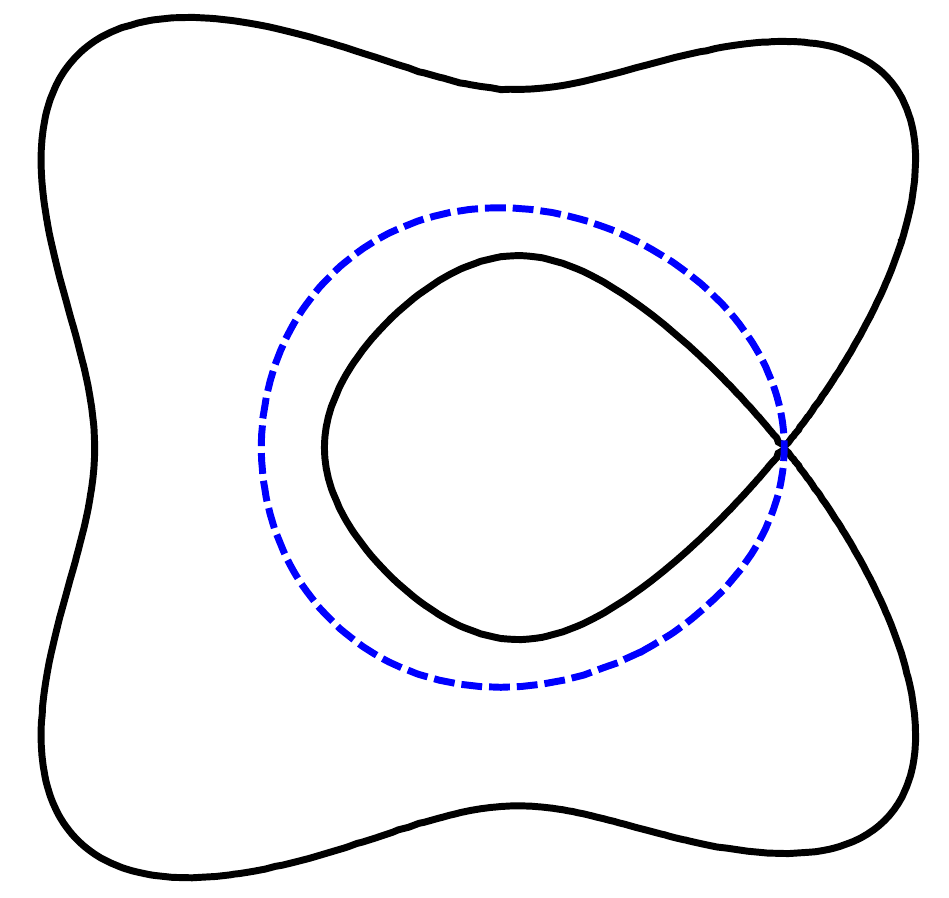} \hspace{5em}
 \includegraphics[height=3.6cm]{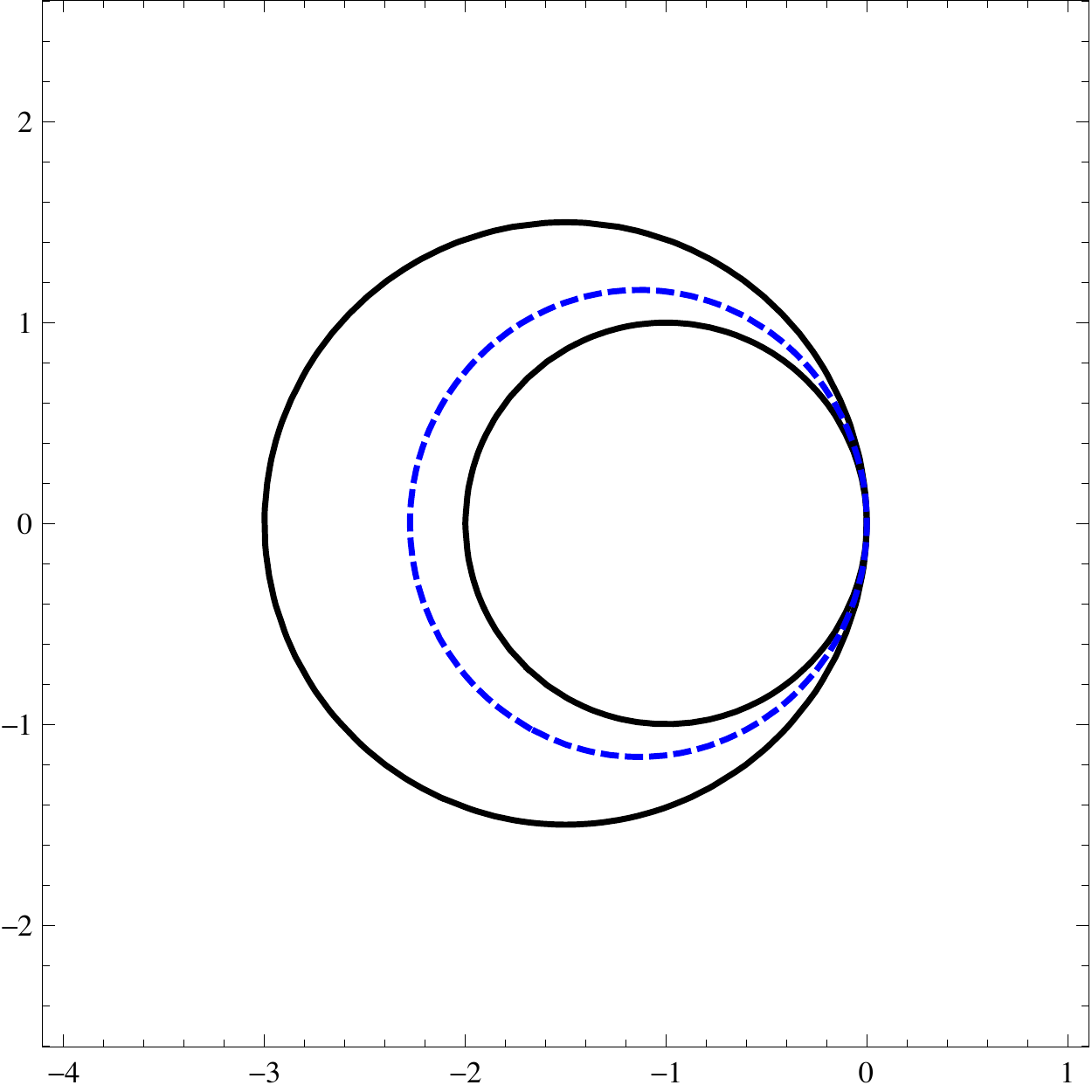}
\caption{Two singular hyperbolic quartics with different dimensions of interlacers. }
\label{fig:singularQuartics}
\end{figure}
 \end{Example}

Theorem~\ref{thm:niceCones} states that $\ol{C(f,e)}$ is a linear slice of the 
cone ${\rm Int}(f,e)$. By taking boundaries of these cones, we recover $\sV(f)$ 
as a linear slice of the algebraic boundary of ${\rm Int}(f,e)$. 

\begin{Def}
  We say that a polynomial $f\in\R[\varx]$ is \emph{cylindrical} if there
  exists an invertible linear change of coordinates $T$ on $\R^{n}$ such that
  $f(Tx)\in\R[x_1,\dots,x_{n-1}]$. 
\end{Def}

\begin{Cor} For non-cylindrical $f$, the map
  $\R^{n}\to\R[x]_{d-1}$ given by $a
  \mapsto D_a f$ is injective and maps
the boundary of $\ol{C(f,e)}$ into the boundary of ${\rm
  Int}(f,e)$. 
If $f$ is irreducible, this map identifies $\mathcal{V}(f)$ with a component of the Zariski closure of the 
boundary of ${\rm Int}(f,e)$ in the plane spanned by $\partial f/\partial x_1, \hdots, \partial f/\partial x_n$. 
\end{Cor}

\begin{proof} Since $f$ is not cylindrical, the $n$ partial
  derivatives $\partial f/\partial x_j$ are linearly independent, so
  that $a\mapsto D_af$ is injective. The claim now follows from taking
  the boundaries of the cones in \eqref{eq:slicedInt}. If $f$ is
  irreducible, then the Zariski closure of the boundary of
  $\ol{C(f,e)}$ is $\mathcal{V}(f)$.
\end{proof}

\begin{Example}  We take the cubic form $f(x,y,z)= (x - y) (x + y) (x + 2 y) - x z^2$,
which is hyperbolic with respect to the point $[1:0:0]$. 
Using the computer algebra system {\tt Macaulay~2}, 
we can calculate the minimal polynomial in  
$\Q[c_{11}, c_{12}, c_{13}, c_{22}, c_{23}, c_{33}]$ 
that vanishes on the boundary of the cone ${\rm Int}(f,e)$. 
Note that any conic
\[q =  c_{11}x^2+c_{12}xy+c_{13}xz+c_{22}y^2+c_{23}yz+c_{33}z^2\]
on the boundary of ${\rm Int}(f,e)$ must have a singular intersection point with 
$\sV(f)$. 
Saturating with the ideal $(x,y,z)$ and eliminating the variables $x,y,z$ from the ideal 
\[(f,q )\; +\;  {\rm minors}_2({\rm Jacobian}(f, q))\]
gives an irreducible polynomial of degree twelve in the six coefficients of $q$.
This hypersurface is the algebraic boundary of ${\rm Int}(f,e)$. 
When we restrict to the three-dimensional subspace given by 
$q = a \frac{\partial f}{\partial x}+b \frac{\partial f}{\partial y}+c \frac{\partial f}{\partial z}$, 
this polynomial of degree twelve factors as
\begin{align*}
a &\cdot f(a,b,c)\cdot (961 a^8 + 5952 a^7 b + 
   11076 a^6 b^2 - 3416 a^5 b^3 - 34770 a^4 b^4 - 31344 a^3 b^5 \\
   & + 14884 a^2 b^6 + 34632 a b^7 + 13689 b^8 - 1896 a^6 c^2 - 
   4440 a^5 b c^2 + 6984 a^4 b^2 c^2 + 25728 a^3 b^3 c^2 \\ &+ 
   15960 a^2 b^4 c^2 - 7560 a b^5 c^2 - 7560 b^6 c^2 + 1074 a^4 c^4 - 
   1680 a^3 b c^4 - 7116 a^2 b^2 c^4 - 2376 a b^3 c^4 \\ &+ 
   2106 b^4 c^4  + 16 a^2 c^6+ 936 a b c^6 - 27 c^8).
   \end{align*}
   One might hope that ${\rm Int}(f,e)$ is also a hyperbolicity cone
   of some hyperbolic polynomial, but we see that this is not the
   case.  Restricting to $c=0$ shows that the polynomial above, unlike $f$,
   is not hyperbolic with respect to $[1:0:0]$.
\end{Example}

\def\cprime{$'$}

\end{document}